\documentclass{gen-j-l}
\usepackage{amsfonts}
\usepackage{geometry}
\usepackage{amsmath}
\usepackage{amssymb}
\usepackage{graphicx}
\newtheorem{theorem}{Theorem}[section]
\newtheorem{lemma}[theorem]{Lemma}
\theoremstyle{definition}
\newtheorem{definition}[theorem]{Definition}
\newtheorem{example}[theorem]{Example}

\theoremstyle{remark}
\newtheorem{remark}[theorem]{Remark}
\numberwithin{equation}{section}
\theoremstyle{plain}

\newtheorem{corollary}{Corollary}

\copyrightinfo{2017}{Carey Caginalp}
\begin{document}
\title[Minimization Solutions to Conservation Laws]{Minimization Solutions to Conservation Laws With Non-smooth and Non-strictly
Convex Flux}
\author{Carey Caginalp}
\email{carey\_caginalp@brown.edu}
\urladdr{http://www.pitt.edu/\symbol{126}careycag/}
\date{\today}

\begin{abstract}
Conservation laws are usually studied in the context of sufficient regularity
conditions imposed on the flux function, usually $C^{2}$ and uniform
convexity. Some results are proven with the aid of variational methods and a
unique minimizer such as Hopf-Lax and Lax-Oleinik. We show that many of these
classical results can be extended to a flux function that is not necessarily
smooth or uniformly or strictly convex. Although uniqueness a.e. of the
minimizer will generally no longer hold, by considering the greatest (or
supremum, where applicable) of all possible minimizers, we can successfully
extend the results. One specific nonlinear case is that of a piecewise linear
flux function, for which we prove existence and uniqueness results. We also
approximate it by a smoothed, superlinearized version parameterized by
$\varepsilon$ and consider the characterization of the minimizers for the
smooth version and limiting behavior as $\varepsilon\downarrow0$ to that of
the sharp, polygonal problem. In proving a key result for the solution in
terms of the value of the initial condition, we provide a stepping stone to
analyzing the system under stochastic processes, which will be explored
further in a future paper.

\end{abstract}
\maketitle

\section{Introduction}

Conservation laws, generally expressed in the form%
\begin{align}
w_{t}+\left(  \mathcal{H}\left(  w\right)  \right)  _{x}  &  =0\text{ in
}\mathbb{R\times}\text{ }\left(  0,\infty\right) \nonumber\\
w\left(  x,0\right)   &  =g^{\prime}\left(  x\right)  \text{ on }%
\mathbb{R}\times\left\{  t=0\right\}  \label{cl}%
\end{align}
and the related Hamilton-Jacobi problem%
\begin{align}
u_{t}+\mathcal{H}\left(  u_{x}\right)   &  =0\text{ in }\mathbb{R\times}\text{
}\left(  0,\infty\right) \nonumber\\
u\left(  x,0\right)   &  =g\left(  x\right)  \text{ on }\mathbb{R}%
\times\left\{  t=0\right\}
\end{align}
for a smooth flux function $\mathcal{H}$ have a wide range of applications,
including modelling shocks mathematical turbulence, and kinetic theory
\cite{BG, CD, CEL, ERS, EV, FM, GR, HP, KR, LA1, LA2, M, MP, MS}. In Section
2, we review some background, based on \cite{EV}, regarding well-established
classical results for the conservation law (\ref{cl}) in the case of a flux
with sufficient regularity conditions. We also show that these results can be
extended in several ways, allowing new, broader application for much of this
well-established theory. For example, we are able to prove several results in
\cite{EV} with the much weaker condition of non-strict convexity assumed on
the flux function rather than uniform convexity.

In addition to relaxing some of the convexity and regularity assumptions, we
consider the specific case of a polygonal flux, a (non-strictly) convex
sequence of piecewise linear segments. It has been studied extensively as a
method of approximation and building up to the smooth case in Dafermos
\cite{D1, D2}. This choice of flux function notably eliminates several of the
assumptions of the usual problem under consideration in that it is (i) not
smooth, (ii) not strictly\ convex, and (iii) not superlinear. The Legendre
transform is also not finite on the entire real line. We consider some of the
results for smooth $\mathcal{H}$ and their possible extension to this case.
Later in this analysis, it will be key to consider a smooth, superlinear
approximation to $H$. We index this approximation by two parameters $\delta$
and $\varepsilon$, corresponding to smoothing and superlinearizing the flux
function, respectively, and denote it by $\mathcal{H}_{\varepsilon,\delta}$.
In Section 3, we prove an existence result for the sharp, polygonal problem,
in addition to several other results, without the properties of being
uniformly convex or superlinear. We also consider the two different types of
minimizers for the sharp problem, both at a vertex of the Legendre transform
$L$ or at a part of $L$ where it is locally differentiable, and demonstrate
how (\ref{introkr}) will hold in various cases with these different species of minimizers.

For the smooth problem, it is well-known (i.e. \cite{EV}) that the minimizer
obtain in closed-form solutions such as Hopf-Lax are unique a.e in $x$ for a
given time $t$. Far more intricate behavior surfaces when one takes a less
smooth flux function, as in our case with a piecewise linear flux. In
particular, the convexity here is no longer uniform and not even strict. As a
result, one can have not only multiple minimizers, but an infinite set of such
points. This involves in-depth analysis of the structure of the minimizers
used in methods such as Hopf-Lax or Lax-Oleinik. In Section 4, by considering
the greatest of these minimizers $y^{\ast}\left(  x,t\right)  $, or its
supremum if not attained, we show that $y^{\ast}\left(  x,t\right)  $ is in
fact increasing in $x$. Further, by carefully considering the relative changes
in this infinite and possibly uncountable number of minima, we rigorously
prove the identity%
\begin{equation}
w\left(  x,t\right)  =g^{\prime}\left(  y^{\ast}\left(  x,t\right)  \right)  .
\label{introkr}%
\end{equation}
This expression relates the solution of the conservation law to the value of
the initial condition evaluated at the point of the minimizer. This is a new
result even under classical conditions and requires a deeper examination of
multiple minimizers in the absence of uniform convexity. We also prove other
results including that the solution is of bounded variation \cite{VO} under
the appropriate assumptions on the initial conditions.

In Section 5, we consider the smoothed and superlinearized flux function
$\mathcal{H}_{\varepsilon,\delta}$. By condensing these two parameters into
one and considering the minimizers of the smooth version, we obtain results
relating to the convergence of these solutions of the smooth flux equation to
the polygonal case.

We can also define a particular kind of uniqueness when constructing the
solution from a certain limit, as we take the aforementioned parameter
$\varepsilon\downarrow0$. We show two uniqueness results, the former using the
smoothing approach of Section 5 and the latter showing that one has uniqueness
under Lipschitz continuity of the initial condition $g^{\prime}\left(
x\right)  $. These results are elaborated on in Section 6.

In Section 7, we consider discontinuous initial conditions. When $H$ is
polygonal and $g^{\prime}$ is piecewise constant with values that match the
break points of $H$, the conservation law becomes a discrete combinatorial
problem. We prove that (\ref{introkr}) is valid, and $w$ can also be obtained
as a limit of solutions to the smoothed problem. This provides a link between
the discrete and continuum conservation laws.

A further application of conservation laws includes the addition of
randomness, such as that in the initial conditions. In doing computations and
analysis relating to these stochastic processes, the identity (\ref{introkr})
will be a key building block. We present some immediate conclusions in Section
8. For example, when applied to Brownian motion, we show that the variance is
the greatest minimizer $y^{\ast}\left(  x,t\right)  $ and increases with $x$
for each $t$. In a second paper, we plan to develop these ideas further.

\section{Classical and New Results For Smooth Flux Functions}

We review briefly the basic theory (see \cite{EV}), and obtain an expression
that will be more useful than the standard results when we relax the
assumptions in order to incorporate polygonal flux. For now we assume that the
flux function $\mathcal{H}\left(  q\right)  :\mathbb{R\rightarrow R}$ is
uniformly convex, continuously differentiable, and superlinear, i.e.,
$\lim_{\left\vert q\right\vert \rightarrow\infty}\mathcal{H}\left(  q\right)
/\left\vert q\right\vert =\infty.$ The Legendre transformation is defined by%
\begin{equation}
\mathcal{L}\left(  p\right)  :=\sup_{q\in\mathbb{R}}\left\{  pq-\mathcal{H}%
\left(  q\right)  \right\}  .
\end{equation}
Here we use script $\mathcal{L}$ and $\mathcal{H}$ to indicate we are
considering the problem with a smooth flux function, and in Section 3 we will
use $L$ and $H$ when considering a piecewise linear flux function.

An initial value problem for the Hamilton-Jacobi problem, on $\mathbb{R}$, is
specified as
\begin{subequations}
\label{clboth}%
\begin{align}
u_{t}+\mathcal{H}\left(  u_{x}\right)   &  =0\label{cla}\\
u\left(  x,0\right)   &  =g\left(  x\right)  \label{clb}%
\end{align}

We call the function $w$ a \textit{weak solution }if it (i) $u\left(
x,t\right)  $ satisfies the initial condition (\ref{cla}) and the equation
(\ref{clb}) a.e. in $\left(  x,t\right)  $ and (ii) (see p. 131 \cite{EV}) for
each $t,$ and a.e. $x$ and $x+z,$ $u\left(  x,t\right)  $ satisfies the
inequality
\end{subequations}
\begin{equation}
u\left(  x+z,t\right)  -2u\left(  x,t\right)  +u\left(  x-z,t\right)  \leq
C\left(  1+\frac{1}{t}\right)  z^{2}\ . \label{onesided}%
\end{equation}
The Hopf-Lax formula is defined by%
\begin{equation}
u\left(  x,t\right)  =\min_{y\in\mathbb{R}}\left\{  t\mathcal{L}\left(
\frac{x-y}{t}\right)  +g\left(  y\right)  \right\}  . \label{sec2hl}%
\end{equation}

The following classical results can be found in \cite{EV}, p. 128 and 145.

\begin{theorem}
\label{Thm2.1}Suppose $\mathcal{H}$ is $C^{2},$ uniformly convex and
superlinear, and $g$ is Lipschitz continuous. Then $u\left(  x,t\right)  $ given by the
Hopf-Lax formula (\ref{sec2hl}) is the unique weak solution to (\ref{clboth}).
\end{theorem}

Now we consider solutions to a related equation, the general conservation law%
\begin{align}
w_{t}+\left(  H\left(  w\right)  \right)  _{x}  &  =0\text{ in }%
\mathbb{R\times}\text{ }\left(  0,\infty\right) \nonumber\\
w\left(  x,0\right)   &  =g^{\prime}\left(  x\right)  \text{ on }%
\mathbb{R}\times\left\{  t=0\right\}  \label{sec2cl}%
\end{align}

\begin{theorem}
Assume that $\mathcal{H}$ is $C^{2},$ uniformly convex, and $g'\in L^{\infty
}\left(  \mathbb{R}\right)  $. Then we have\newline\newline(i) For each $t>0$
and for all but countably many values $x\in\mathbb{R}$, there exists a unique
point $y\left(  x,t\right)  $ such that%
\begin{equation}
\min_{y\in\mathbb{R}}\left\{  t\mathcal{L}\left(  \frac{x-y}{t}\right)
+g\left(  y\right)  \right\}  =t\mathcal{L}\left(  \frac{x-y\left(
x,t\right)  }{t}\right)  +g\left(  y\left(  x,t\right)  \right)
\end{equation}
(ii) The mapping $x\rightarrow y\left(  x,t\right)  $ is
nondecreasing.\newline\newline(iii) For each $t>0$, the function $w$ defined
by%
\begin{equation}
w\left(  x,t\right)  :=\frac{\partial}{\partial x}\left[  \min_{y\in
\mathbb{R}}\left\{  t\mathcal{L}\left(  \frac{x-y}{t}\right)  +g\left(
y\right)  \right\}  \right]  \label{Thm2.2eq}%
\end{equation}
is in fact given by%
\[
w\left(  x,t\right)  =\left(  \mathcal{H}^{\prime}\right)  ^{-1}\left(
\frac{x-y\left(  x,t\right)  }{t}\right)
\]

\end{theorem}

To illuminate the notion of a weak solution, we briefly describe the
motivation of the definition. Nominally, if we had a smooth function $u$ that
satisfied (\ref{cla}) everywhere in $\left(  x,t\right)  $ and the initial
condition (\ref{clb}) then we could multiply (\ref{cla}) by the spatial
derivative of the test function $\phi\in C_{c}^{\infty}\left(  \mathbb{R\times
}\left(  0,\infty\right)  \right)  $ and integrate by parts to obtain
\begin{equation}
\int_{0}^{\infty}\left\{  \int_{-\infty}^{\infty}u\phi_{xt}+\mathcal{H}\left(
u_{x}\right)  \phi_{x}dx\right\}  dt+\int_{-\infty}^{\infty}g\phi_{x}%
|_{t=0}dx=0\ .
\end{equation}
Now we let $w:=u_{x}$ and integrate by parts in the $x$ variable, (see
\cite{EV} p. 148 for details and conditions). Note that $u\left(  x,t\right)
$ is by assumption differentiable a.e. The test function is differentiable at
all points, and so the product rule applies outside of a set of measure zero.
Hence, one can integrate, and one then has%
\begin{equation}
\int_{0}^{\infty}\left\{  \int_{-\infty}^{\infty}w\phi_{t}+\mathcal{H}\left(
w\right)  \phi_{x}dx\right\}  dt+\int_{-\infty}^{\infty}g^{\prime}\phi
|_{t=0}dx=0\ . \label{testfn}%
\end{equation}
We now say that $w$ is a weak solution to the conservation law if it satisfies
(\ref{testfn}) for all test functions with compact support.

\begin{remark}
\label{Rk uniquecl}From classical theorems, we also know that under the
conditions that $g'$ is continuous and $\mathcal{H}$ is $C^{2}$ and
superlinear, we have a unique weak solution to (\ref{Thm2.2eq}) that is an
integral solution to the conservation law (\ref{sec2cl})$.$ However, at this
stage we do not know if there are other solutions to (\ref{sec2cl}) arising
from a different perspective, where $g$ is a differentiable function.
\end{remark}

In order to obtain a unique solution to the conservation law, one imposes an
additional entropy condition and makes the following definition.

\begin{definition}
We call $w\left(  x,t\right)  $ an \textit{entropy solution} to (\ref{sec2cl})
if: (i) it satisfies (\ref{testfn}) for all test functions $\phi
:\mathbb{R\times}[0,\infty)\rightarrow\mathbb{R}$ that have compact support
and (ii) for a.e. $x\in\mathbb{R}$, $t>0,$ $z>0$, we have
\begin{equation}
w\left(  x+z,t\right)  -w\left(  x,t\right)  \leq C\left(  1+\frac{1}%
{t}\right)  z\ .
\end{equation}

\end{definition}

In order to prove that $w=u_{x}$ is the unique solution to (\ref{sec2cl}), we
note the following: In Theorem 1, p. 145 of \cite{EV} it suffices for the
initial condition to be continuous. In the theorem, the only use of the
$\mathcal{L}^{\infty}\left(  \mathbb{R}\right)  $ condition is that its
integral is differentiable a.e. which is certainly guaranteed by the continuity.

Under the assumptions of Theorem 1, the Lemma of p. 148 of \cite{EV} states
that, with $G:=\left(  \mathcal{H}^{\prime}\right)  ^{-1},$ the function
$w=u_{x}$, i.e.,
\begin{equation}
w\left(  x,t\right)  =\partial_{x}u\left(  x,t\right)  =\partial_{x}\min
_{y\in\mathbb{R}}\left\{  t\mathcal{L}\left(  \frac{x-y}{t}\right)  +g\left(
y\right)  \right\}  =G\left(  \frac{x-y^{\ast}\left(  x,t\right)  }{t}\right)
\label{hl2}%
\end{equation}
satisfies the one-sided inequality%
\begin{equation}
w\left(  x+z,t\right)  -w\left(  x,t\right)  \leq\frac{C}{t}z\ . \tag{7}%
\end{equation}

Once we have established that $w$ is an entropy solution, the uniqueness of
the entropy solution (up to a set of measure zero) is a basic result that is
summarized in \cite{EV} (Theorem 3, p 149):

\bigskip

\begin{theorem}
\label{Thm 2.2}Assume $H$ is convex and $C^{2}$. Then there exists (up to a
set of measure $0$), at most one entropy solution of (\ref{sec2cl}).
\end{theorem}

\bigskip

Note that one only needs $g'$ to be $\mathcal{L}^{\infty}$ in this theorem. One
has then the classical result:

\begin{theorem}
\label{Thm 2.3}Assume that $\mathcal{H}$ is $C^{2},$ superlinear and uniformly
convex. Then the function $w\left(  x,t\right)  $ given by (\ref{hl2}) is the
unique entropy solution to the conservation law (\ref{testfn})$.$
\end{theorem}

Note that we need the uniformly convexity assumption in order that the
one-sided condition holds, which in turn is necessary for the uniqueness.

\bigskip

A classical result is that if $y\left(  x,t\right)  $ is defined as a
minimizer of
\begin{equation}
Q\left(  y;x,t\right)  :=t\mathcal{L}\left(  \frac{x-y}{t}\right)  +g\left(
y\right)
\end{equation}
then it is unique and the mapping $x\mapsto y\left(  x,t\right)  $ is
non-decreasing, and hence, continuous except at countably many points $x$ (for
each $t$) and differentiable a.e., in $x$ for each $t.$ The Lax-Oleinik
formula above, which expresses the solution $w$ to the conservation law as a
function of $\left(  \mathcal{H}^{\prime}\right)  ^{-1}$.

This formula, of course, utilizes the fact that $\mathcal{H}^{\prime}$ is
strictly increasing, i.e., that $\mathcal{H}\in C^{2}$ and uniformly convex.
Using similar ideas, we present a more useful formula that will be shown in
later theorems to be valid even when the inverse of $\mathcal{H}^{\prime}$
does not exist. For these theorems we need the following notion to express the
argument of a minimizer.

\begin{definition}
Let $B$ be a measurable set and suppose that there is a unique minimizer
$y^{\ast}$ for a quantity $Q\left(  y\right)  $ such that%
\[
Q\left(  y^{\ast}\right)  =\min_{y\in B}Q\left(  y\right)  .
\]
Define the function $\arg$ to mean that%
\[
y^{\ast}=:\arg\min_{y\in B}Q\left(  y\right)  .
\]
In the case that the minimum is achieved over some collection of points in
$B$, denote by $\arg^{+}$ the supremum of all such points, regardless of
whether the supremum of this set is a minimizer itself.
\end{definition}

\begin{theorem}
\label{Thm 2.4}Let $\mathcal{H}\in C^{2}$ and convex and $g\in C^{1}$. Suppose
that for each $\left(  x,t\right)  ,$ the quantity
\begin{equation}
y^{\ast}\left(  x,t\right)  =\inf_{y\in\mathbb{R}}\left\{  t\mathcal{L}\left(
\frac{x-y}{t}\right)  +g\left(  y\right)  \right\}
\end{equation}
is well-defined, finite, and \textbf{unique}. Then
\begin{equation}
\mathcal{L}^{\prime}\left(  \frac{x-y^{\ast}\left(  x,t\right)  }{t}\right)
=g^{\prime}\left(  y^{\ast}\left(  x,t\right)  \right)
\end{equation}
and $w\left(  x,t\right)  :=\partial_{x}\min_{y\in\mathbb{R}}\left\{
t\mathcal{L}\left(  \frac{x-y}{t}\right)  +g\left(  y\right)  \right\}  $ is
given by%
\begin{equation}
w\left(  x,t\right)  =g^{\prime}\left(  y^{\ast}\left(  x,t\right)  \right)
\ .
\end{equation}

\end{theorem}

\begin{proof}
[Proof of Theorem \ref{Thm 2.4}]From Section 3.4, Thm 1 of \cite{EV}, we know
that a minimizer of $t\mathcal{L}\left(  \frac{x-y}{t}\right)  +g\left(
y\right)  $ (if unique) is differentiable a.e. in $x.$ We then have the
following calculations.

Since we are assuming that $\inf_{y\in\mathbb{R}}\left\{  t\mathcal{L}\left(
\frac{x-y}{t}\right)  +g\left(  y\right)  \right\}  >-\infty$ and both
$\mathcal{L}$ and $g$ are differentiable, there exists a minimizer. Since
$\mathcal{L}$ and $g$ are differentiable, for any potential minimizer one has
the identity%
\begin{equation}
0=\partial_{y}\left\{  t\mathcal{L}\left(  \frac{x-y}{t}\right)  +g\left(
y\right)  \right\}
\end{equation}
so that (for a.e. $x$) at a minimum, $y^{\ast}\left(  x,t\right)  ,$ one has
\begin{equation}
-f\left(  y^{\ast},x,t\right)  :=\mathcal{L}^{\prime}\left(  \frac{x-y^{\ast
}\left(  x,t\right)  }{t}\right)  =g^{\prime}\left(  y^{\ast}\left(
x,t\right)  \right)  .
\end{equation}
We have then at any point $x$ where $y^{\ast}\left(  x,t\right)  $ is
differentiable,
\begin{align}
w\left(  x,t\right)   &  :=\partial_{x}\min_{y\in\mathbb{R}}\left\{
t\mathcal{L}\left(  \frac{x-y}{t}\right)  +g\left(  y\right)  \right\}
\label{Thm 2.4 algebra}\\
&  =\partial_{x}\left\{  t\mathcal{L}\left(  \frac{x-y^{\ast}\left(
x,t\right)  }{t}\right)  +g\left(  y^{\ast}\left(  x,t\right)  \right)
\right\} \nonumber\\
&  =t\mathcal{L}^{\prime}\left(  \frac{x-y^{\ast}\left(  x,t\right)  }%
{t}\right)  \cdot\left(  \frac{-\partial_{x}y^{\ast}\left(  x,t\right)  }%
{t}+1\right)  +g^{\prime}\left(  y^{\ast}(x,t)\right)  \partial_{x}y^{\ast
}\left(  x,t\right)  .
\end{align}
The previous identity implies cancellation of the first and third terms,
yielding%
\begin{equation}
w\left(  x,t\right)  =g^{\prime}\left(  y^{\ast}\left(  x,t\right)  \right)
\ \ \ \ a.e.\ x\in\mathbb{R}\text{ }for\ each\ t>0.
\end{equation}

\end{proof}

Note that the uniqueness of the minimizer is used in the second line of
(\ref{Thm 2.4 algebra}). If there were two minimizers, for example, then as we
vary $x,$ one of the minima might decrease more rapidly, and that would be the
relevant minimum for the $x$ derivative.

We now explore the case with two minimizers. Using the notation $t\mathcal{L}%
\left(  \frac{x-y}{t}\right)  =:f\left(  y;x,t\right)  $ as defined above, we
note that whenever we have a minimum of $f\left(  y;x,t\right)  +g\left(
y\right)  $ at some $y_{0}$ we must have%
\begin{equation}
\partial_{y}f\left(  y_{0};x,t\right)  +g^{\prime}\left(  y_{0}\right)  =0.
\end{equation}
We are interested in computing $w\left(  x,t\right)  =\partial_{x}\min
_{y\in\mathbb{R}}\left\{  f\left(  y;x,t\right)  +g\left(  y\right)  \right\}
.$ Suppose that there are two distinct minima, $\hat{y}_{0}$ and $\tilde
{y}_{0},$ with $\hat{y}_{0}<\tilde{y}_{0}$ at some point $x_{0}.$ Then we can
define $\hat{y}\left(  x,t\right)  $ and $\tilde{y}\left(  x,t\right)  $ as
distinct local minimizers that are differentiable in $x,$ and satisfy
\begin{equation}
\lim_{x\rightarrow x_{0}}\hat{y}\left(  x,t\right)  =\hat{y}_{0}%
\ \ and\ \ \lim_{x\rightarrow x_{0}}\tilde{y}\left(  x,t\right)  =\tilde
{y}_{0}\ .
\end{equation}
Then as we vary $x$, the minima will shift vertically and horizontally. The relevant minima are those that have the largest downward shift, as the others immediately cease to be minima.

This means that%
\begin{equation}
w\left(  x,t\right)  =\min\left\{  \partial_{x}\left[  f\left(  \hat{y}\left(
x,t\right)  ;x,t\right)  +g\left(  \hat{y}\left(  x,t\right)  \right)
\right]  ,\partial_{x}\left[  f\left(  \tilde{y}\left(  x,t\right)
;x,t\right)  +g\left(  \tilde{y}\left(  x,t\right)  \right)  \right]
\right\}  |_{x=x_{0}}\ .
\end{equation}
Then, as the calculations in the proof of the theorem above show, one has%
\begin{align}
w\left(  x,t\right)   &  =\min\left\{  g^{\prime}\left(  \hat{y}_{0}\right)
,g^{\prime}\left(  \tilde{y}_{0}\right)  \right\} \nonumber\\
&  =\min\left\{  -\partial_{y}f\left(  \hat{y}_{0},x_{0},t\right)
,-\partial_{x}f\left(  \tilde{y}_{0};x,t\right)  \right\}  .
\end{align}
Since we are assuming that $\hat{y}_{0}<\tilde{y}_{0}$ and $f^{\prime}$ is
increasing, we see that the minimum of these two is $-\partial_{y}f\left(
\tilde{y}_{0};x_{0},t\right)  ,$ yielding,%
\begin{equation}
w\left(  x,t\right)  =-\partial_{y}f\left(  \tilde{y}_{0};x_{0},t\right)
=g^{\prime}\left(  \tilde{y}_{0}\right)  .
\end{equation}

Now suppose that for fixed $\left(  x_{0},t\right)  $ we have a set of
minimizers $\left\{  y_{\alpha}\right\}  $ with $\alpha\in A$ for some set
$A$. Should $A$ consist of a finite number of elements, an elementary
extension of the above argument generalizes the result to the maximum of these minimizers.

Next, suppose that the set has an infinite number of members. The case where
the supremum of this set is $+\infty$ is degenerate and will be excluded by
our assumptions. Thus, assume that for a given $\left(  x,t\right)  ,$ the set
$\left\{  y_{\alpha}\right\}  $ is bounded, and call its\ supremum $y^{\ast}.$
Then either $y^{\ast}\in A,$ i.e., it must be a minimizer, or there is a
sequence $\left\{  y_{j}\right\}  $ in $A$ converging to $y^{\ast}.$ If
$y^{\ast}\not \in A$, then we have, similar to the assertion above, the
identity%
\begin{equation}
w\left(  x,t\right)  =\inf_{\alpha\in A}\left\{  -\partial_{y}f\left(
y_{\alpha};x,t\right)  \right\}  .
\end{equation}
Since $f\in C^{2}$ and $y_{j}\rightarrow y^{\ast}$ we see that
\begin{equation}
w\left(  x,t\right)  =-\partial_{y}f\left(  y^{\ast};x,t\right)  =g^{\prime
}\left(  y^{\ast}\right)  . \label{Sec2supmin}%
\end{equation}
Note that (\ref{Sec2supmin}) is valid whether or not $y^{\ast}$ is a minimizer.

Suppose that $f\in C^{2}$ is convex and that we have a continuum of minimizers
again. Suppose further that $f^{\prime}$ is nondecreasing, and there is an
interval $\left[  a,b\right]  $ of minimizers of $\left\{  f\left(
y;x,t\right)  +g\left(  y\right)  \right\}  .$ Note that the form of $f$ is
such that we can write it as%
\begin{equation}
f\left(  y;x,t\right)  =\hat{f}\left(  y-x\right)
\end{equation}
with $\hat{f}$ increasing. We can perform a calculation similar to the ones
above by drawing the graphs of $\hat{f}$ and $g$ as a function of $y$ at
$x_{0}$ as follows:%
\begin{align}
w\left(  x,t\right)   &  =\partial_{x}\min_{y\in\left[  a,b\right]  }\left\{
f\left(  y;x_{0}\right)  +g\left(  y\right)  \right\} \nonumber\\
&  =\lim_{\delta\rightarrow0}\frac{\min_{y\in\left[  a,b\right]  }\left\{
\hat{f}\left(  y-x_{0}-\delta\right)  +g\left(  y\right)  \right\}
-\min_{y\in\left[  a,b\right]  }\left\{  \hat{f}\left(  y-x_{0}\right)
+g\left(  y\right)  \right\}  }{\delta}.
\end{align}
We are assuming that there is an interval $y\in\left[  a,b\right]  $ of
minimizers, such that $f\left(  y-x_{0}\right)  +g\left(  y\right)  =C_{1}$
for some constant $C_{1}.$ This means $f\ ^{\prime}\left(  y-x_{0}\right)
=g^{\prime}\left(  y\right)  .$ Since $C_{1}$ occurs on both parts of the
subtraction, we can drop it. Using the mean value theorem we have%
\begin{equation}
\hat{f}\left(  y-x_{0}-\delta\right)  =\hat{f}\left(  y-x_{0}\right)
-\delta\hat{f}\ ^{\prime}\left(  \zeta\right)
\end{equation}
where $\zeta$ is between $y-x_{0}$ and $y-x_{0}-\delta.$ Using the identity
$\hat{f}\ ^{\prime}\left(  y-x_{0}\right)  =g^{\prime}\left(  y\right)  $ we
can write%
\begin{align}
w\left(  x,t\right)   &  =\lim_{\delta\rightarrow0}\frac{\min_{y\in\left[
a,b\right]  }\left\{  -\delta\hat{f}\ ^{\prime}\left(  \zeta\right)  \right\}
-0}{\delta}\nonumber\\
&  =\lim_{\delta\rightarrow0}\min_{y\in\left[  a,b\right]  }\left\{  -\hat
{f}\ ^{\prime}\left(  \zeta\right)  \right\}  =-\hat{f}\ ^{\prime}\left(
b-x_{0}\right)  =-f\ ^{\prime}\left(  b\right)
\end{align}
since $f^{\prime}$ is nondecreasing, and the minimum of $-f^{\prime}\left(
y\right)  $ is attained at the rightmost point.

Although we have only considered the cases where the set of minimizers is
countable or an interval, this argument suffices for the general case. Indeed,
the set $A$ of minimizers will be measurable. If it has finite measure, it can
be expressed as a countable union of disjoint closed intervals $A_{j}$, i.e.
$A=\cup_{j=1}^{\infty}A_{j}$. It is then equivalent to apply the argument for
the countable set of minimizers to the points $y_{j}=\sup A_{j}$ and proceed
as above. To illustrate these ideas, consider the following example.

\begin{example}
\label{ExIntervalSmooth}Let $f\left(  y;x\right)  :=\left(  x-y\right)  ^{2}$
and $g\left(  y\right)  :=-y^{2}$ for $y\in\left[  a,b\right]  $ and increase
rapidly outside of $\left[  a,b\right]  $, suppressing $t.$ We have $f\left(
y;0\right)  +g\left(  y\right)  =y^{2}-y^{2}=0$ so all points in $\left[
a,b\right]  $ are minimizers (see Figure \ref{IMx2}). We want to calculate%
\begin{equation}
\partial_{x}\min_{y\in\left[  a,b\right]  }\left\{  f\left(  y;x\right)
+g\left(  y\right)  \right\}  |_{x=0}\ ,\label{count_min1}
\end{equation}
i.e.,%
\begin{align}
&  \lim_{\delta\rightarrow0}\frac{\min_{y\in\left[  a,b\right]  }\left\{
f\left(  y;\delta\right)  +g\left(  y\right)  \right\}  -\min_{y\in\left[
a,b\right]  }\left\{  f\left(  y;0\right)  +g\left(  y\right)  \right\}
}{\delta}\nonumber\\
&  =\lim_{\delta\rightarrow0}\left\{  \delta^{-1}\min_{y\in\left[  a,b\right]
}\left\{  \left(  y-\delta\right)  ^{2}-y^{2}\right\}  \right\}  =-2b.
\end{align}
I.e., $w\left(  x,t\right)  :=\partial_{x}\min_{y\in\left[  a,b\right]
}\left\{  f\left(  y;x\right)  +g\left(  y\right)  \right\}  |_{x=0}$ is given
by $-\partial_{y}f\left(  y;x\right)  $ at the rightmost point of the interval
$\left[  a,b\right]  $:
\begin{equation}
-\partial_{y}f\left(  y;0\right)  =-2b\ at~y=b. \label{count_min2}
\end{equation}
Note that by continuity, we have the same conclusion if the interval is open
at the right endpoint $b$.
\end{example}

\bigskip
\begin{figure}
[ptb]
\begin{center}
\includegraphics[width=\linewidth]{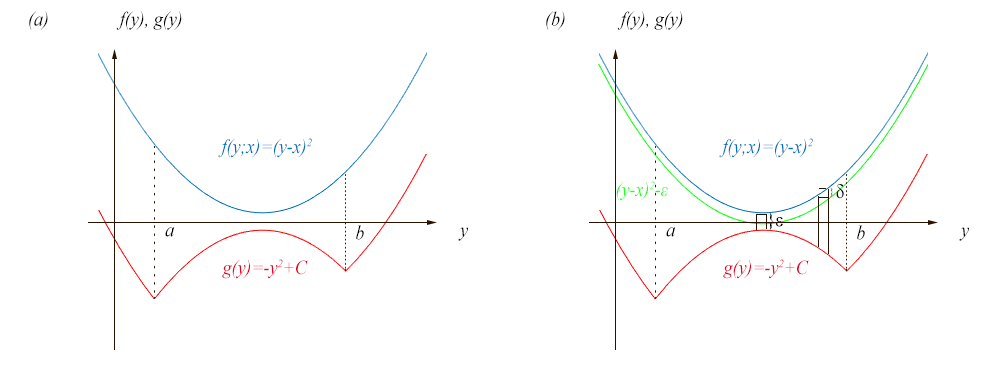}
\caption{Consider the case when $f\left(  y;x\right)  $ takes the parabolic
form $\left(  x-y\right)  ^{2}$: (a) In some pathological cases, the initial
condition may coincide in such a way that for an entire interval $\left[
a,b\right]  $, the infimum on the right-hand side of (\ref{Thm2.5min}) is
achieved. In Example \ref{ExIntervalSmooth}, we illustrate the application of
Theorem \ref{Thm 2.5}; (b) Observe that when $f$ is shifted by a small amount,
there is a varying impact on the continuum of minimizers (where here
$\delta<\varepsilon$), affirming the sensitivity of this special case.}
\label{IMx2}%
\end{center}
\end{figure}
%EndExpansion

Using the calculations in (\ref{count_min1})-(\ref{count_min2}), we can improve Theorem \ref{Thm 2.4} above by
removing the "unique minimizer"\ restriction.

\begin{theorem}
\label{Thm 2.5}Let $\mathcal{H}\in C^{2}$ and convex and $g\in C^{1}$. Suppose
that for each $\left(  x,t\right)  $, the quantity%
\begin{equation}
y^{\ast}\left(  x,t\right)  =\arg^{+}\inf_{y\in\mathbb{R}}\left\{
t\mathcal{L}\left(  \frac{x-y}{t}\right)  +g\left(  y\right)  \right\}
\label{Thm2.5min}%
\end{equation}
is well-defined (finite). Then%
\begin{equation}
\mathcal{L}^{\prime}\left(  \frac{x-y^{\ast}\left(  x,t\right)  }{t}\right)
=g^{\prime}\left(  y^{\ast}\left(  x,t\right)  \right)
\end{equation}
and $w\left(  x,t\right)  :=\partial_{x}\min_{y\in\mathbb{R}}\left\{
t\mathcal{L}\left(  \frac{x-y}{t}\right)  +g\left(  y\right)  \right\}  $ is
given by%
\begin{equation}
w\left(  x,t\right)  =g^{\prime}\left(  y^{\ast}\left(  x,t\right)  \right)
\ .
\end{equation}

\end{theorem}

\begin{remark}
The condition (\ref{Thm2.5min}) is not difficult to satisfy, as we simply need
$g$ to be well-defined on some interval where $\mathcal{L}$ is finite.
\end{remark}

\begin{remark}
Theorem \ref{Thm 2.5} improves upon the classic theorem, which requires
uniform convexity. By utilizing the concept of the greatest minimizer
$y^{\ast}$, we are able to deal with non-unique minimizers and obtain an
expression for the solution to the conservation law using only convexity and
not requiring uniform or strict convexity.
\end{remark}

\section{Existence of Solutions For Polygonal (Non-Smoothed)\ Flux}

We use the general theme of \cite{EV} and adapt the proofs to polygonal flux
(i.e., not smooth or superlinear). We define the Legendre transform without
the assumption of superlinearity on the flux function $H$. Although this
causes its Legendre transform $L$ to be infinite for certain points, one can
still perform computations and prove results close to those of the previous
section under these weaker assumptions, as $L$ is used in the context of
minimization problems..

The first matter is to make sure that we have the key theorem that $H$ and $L$
are Legendre transforms of one another. We do not need to use any of the
theorems that rely on superlinearity. We only assume that $L$ is Lipschitz
continuous, which follows from the definition of $H$. We also assume that
$g$ (the initial condition for the Hamilton-Jacobi equation) is Lipschitz on
specific finite intervals.

Throughout this section, we make the assumption that $H\left(  q\right)  $ is
polygonal convex with the line segments having slopes $m_{1}$ at the left and
$m_{N+1}$ at the right, with break points $c_{1}<c_{2}<...<c_{N}$, with
$c_{1}<0<c_{N}$. The Legendre transform, $L\left(  q\right)  $, defined below
is then also polygonal and convex on $\left[  m_{1},m_{N+1}\right]  $ and
infinite elsewhere. We will assume $m_{1}<0<m_{N+1}$. We illustrate this flux
function and some of the properties of the Legendre transform in Figure
\ref{PWF}.

\begin{definition}
We define the usual Legendre transform, denoted by $L\left(  p\right)  $, as
follows:%
\[
L\left(  p\right)  :=\sup_{q\in\mathbb{R}}\left\{  pq-H\left(  q\right)
\right\}
\]

\end{definition}

A computation shows that this is a convex polygonal shape such that $L\left(
p\right)  <\infty$ if and only if $p\in\left[  m_{1},m_{N}+1\right]  .$ It has
break points at $m_{1}<m_{2}<...<m_{N+1}$ and slopes $c_{1},c_{2},..,c_{N}$.
The last break point of $L$ is at $m_{N+1}$ where the slope and $L\left(
m_{N+1}\right)  $ become infinite. Note that $L$ is Lipschitz on $\left[
m_{1},m_{N+1}\right]  .$
\begin{figure}
[h]
\begin{center}
\includegraphics[width=\linewidth]{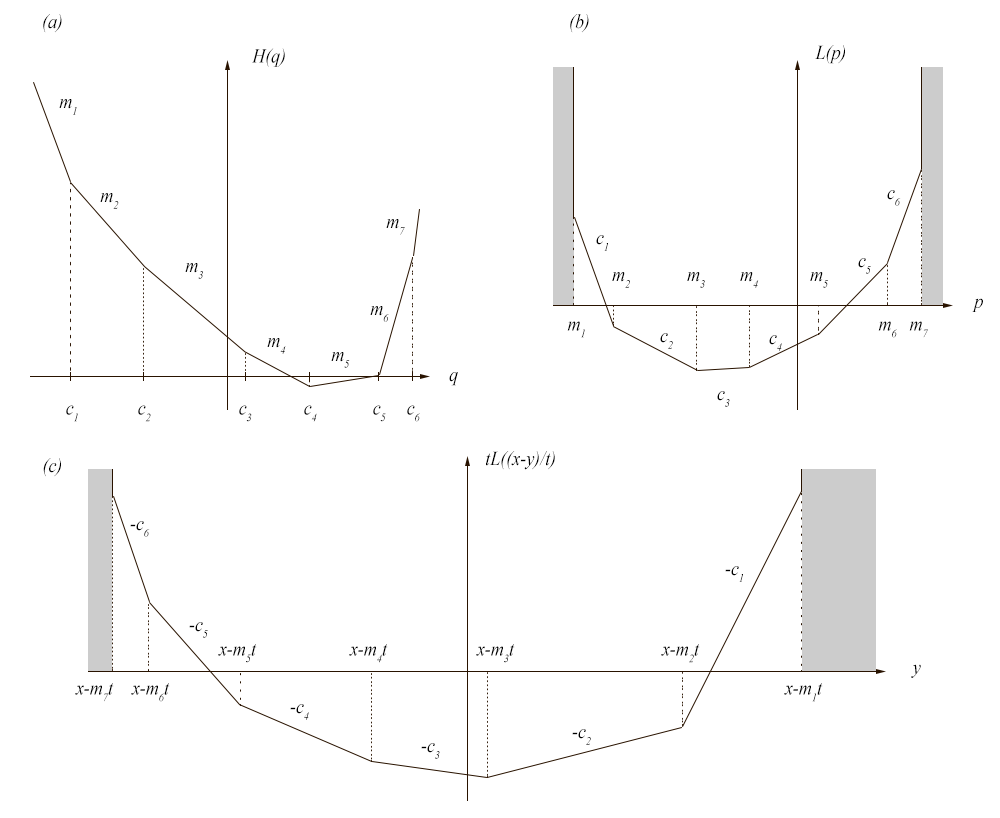}
\caption{(a)\ Illustration of the piecewise linear flux function $H\left(
q\right)  $; (b) The Legendre transform $L\left(  p\right)  $ of the function
$H\left(  q\right)  $; (c) The Legendre transform evaluated at a point as used
in the Hopf-Lax minimization problem. Note that the value of the Legendre
transform is infinite in the shaded regions; however, this pathology is
mitigated since one generally is interested in a minimization problem, which
limits the domain of this operation to a certain interval.}
\label{PWF}%
\end{center}
\end{figure}
%EndExpansion

\begin{lemma}
\label{Prop1}(Duality) Let $L\left(  p\right)  $ be as defined above (with
$L\left(  p\right)  <\infty$ iff $p\in\left[  m_{1},m_{N+1}\right]  $ ). Then
the Legendre transform $L\left(  p\right)  $ of $H\left(  p\right)  $ and the
flux function $H\left(  q\right)  $ itself satisfy the following duality
condition:%
\[
L^{\ast}\left(  q\right)  :=\sup_{p\in\mathbb{R}}\left\{  pq-L\left(
q\right)  \right\}  =H\left(  q\right)  .
\]

\end{lemma}

In other words, if we define $L\left(  p\right)  $ as polygonal, convex
function between $p\in\left[  m_{1},m_{N+1}\right]  $ as in Figure \ref{PWF},
with $L\left(  p\right)  =\infty$ for $p\not \in \left[  m_{1},m_{N+1}\right]
$ then the operation $\sup_{p\in\mathbb{R}}\left\{  pq-L\left(  q\right)
\right\}  $ yields the function $H$ defined above. One can then prove a set of
lemmas that are the analogs of those in Section 3.4 of \cite{EV}. The proofs
are adapted in order to handle potentially infinite values.

\begin{lemma}
\label{Lem A}Suppose $L$ is defined as above and $g$ is Lipschitz continuous
on bounded sets. Then $u$ defined by the Hopf-Lax formula is Lipschitz
continuous in $x,$ independently of $t.$ Moreover,%
\[
\left\vert u\left(  x+z,t\right)  -u\left(  x,t\right)  \right\vert \leq
Lip\left(  g\right)  \left\vert z\right\vert .
\]

\end{lemma}

\begin{proof}
[Proof of Lemma \ref{Lem A}]Fix $t>0,$ $x,\hat{x}\in\mathbb{R}$ . Choose
$y\in\mathbb{R}$ (depending on $x,t$) such that
\begin{equation}
u\left(  x,t\right)  =\min_{z}\left\{  tL\left(  \frac{x-z}{t}\right)
+g\left(  z\right)  \right\}  =tL\left(  \frac{x-y}{t}\right)  +g\left(
y\right)  . \label{A1}%
\end{equation}
The minimum is attained since both $L$ and $g$ are continuous. Note that while
there may be values of $\left(  x-z\right)  /t$ such that $L\left(  \left(
x-z\right)  /t\right)  =\infty,$ these are irrelevant, as there are some
finite values, and $\frac{x-y}{t}$ will be one of those. Now use (\ref{A1}) to
write%
\begin{equation}
u\left(  \hat{x},t\right)  -u\left(  x,t\right)  =\inf_{\tilde{z}}\left\{
tL\left(  \frac{\hat{x}-\tilde{z}}{t}\right)  +g\left(  \tilde{z}\right)
\right\}  -tL\left(  \frac{x-y}{t}\right)  -g\left(  y\right)  \ . \label{A2}%
\end{equation}

We define $z:=\hat{x}-x+y$ such that
\begin{equation}
\frac{x-y}{t}=\frac{\hat{x}-z}{t} \label{A3}%
\end{equation}
and substitute this $z$ in place of $\tilde{z}$ in (\ref{A2}) which can only
increase the RHS. This yields the inequality%
\begin{align}
u\left(  \hat{x},t\right)  -u\left(  x,t\right)   &  \leq tL\left(  \frac
{\hat{x}-z}{t}\right)  +g\left(  z\right)  -tL\left(  \frac{x-y}{t}\right)
-g\left(  y\right) \nonumber\\
&  =tL\left(  \frac{x-y}{t}\right)  +g\left(  \hat{x}-x+y\right)  \ -tL\left(
\frac{x-y}{t}\right)  -g\left(  y\right) \nonumber\\
&  =g\left(  \hat{x}-x+y\right)  -g\left(  y\right)  .
\end{align}
Using the assumption that $g$ is Lipschitz on bounded sets, one has
\begin{equation}
u\left(  \hat{x},t\right)  -u\left(  x,t\right)  \leq Lip\left(  g\right)
\cdot\left\vert \hat{x}-x\right\vert \ . \label{A4}%
\end{equation}

Note that in obtaining this inequality, $x$ and $\hat{x}$ were arbitrary
(without any assumption on order). Hence, we can interchange them. I.e., we
start by defining $y$ such that, instead of (\ref{A1}), it satisfies
\begin{equation}
u\left(  \hat{x},t\right)  =tL\left(  \frac{\hat{x}-y}{t}\right)  +g\left(
y\right)  .
\end{equation}
Thus, we obtain the same inequality as (\ref{A4}) with the $x$ and $\hat{x}$
interchanged, yielding%
\begin{equation}
\left\vert u\left(  \hat{x},t\right)  -u\left(  x,t\right)  \right\vert \leq
Lip\left(  g\right)  \cdot\left\vert \hat{x}-x\right\vert \ .\ \label{A5}%
\end{equation}

\end{proof}

\bigskip

\begin{lemma}
\label{Lem 1}Suppose $L$ is defined as above and $g$ is Lipschitz continuous
on bounded sets$.$ For each $x\in\mathbb{R}$ and $0\leq s<t$ we have
\[
u\left(  x,t\right)  =\min_{y\in\mathbb{R}}\left\{  \left(  t-s\right)
L\left(  \frac{x-y}{t-s}\right)  +u\left(  y,s\right)  \right\}  .
\]

\end{lemma}

\begin{proof}
[Proof of Lemma \ref{Lem 1}]A. Fix $y\in\mathbb{R}$, $0<s<t.$ Since $u$ and
$L$ are continuous, the minimum is attained on the interval $\left[
m_{1},m_{N+1}\right]  $ where $L$ is finite. Thus we can find $z\in\mathbb{R}$
such that
\begin{equation}
u\left(  y,s\right)  =sL\left(  \frac{y-z}{s}\right)  +g\left(  z\right)  .
\label{L11}%
\end{equation}
Note that since $z$ is the minimizer of $L\left(  \left(  y-z\right)
/s\right)  $, we know that $L\left(  \left(  y-z\right)  /s\right)  $ is finite.

By convexity of $L$ we can write%
\[
\frac{x-z}{t}=\left(  1-\frac{s}{t}\right)  \left(  \frac{x-y}{t-s}\right)
+\frac{s}{t}\left(  \frac{y-z}{s}\right)
\]%
\begin{equation}
L\left(  \frac{x-z}{t}\right)  \leq\left(  1-\frac{s}{t}\right)  L\left(
\frac{x-y}{t-s}\right)  +\frac{s}{t}L\left(  \frac{y-z}{s}\right)  .
\label{L12}%
\end{equation}
Next, we have from our basic assumption that $u\left(  x,t\right)  $ is
defined by the Hopf-Lax formula, the identity%
\begin{equation}
u\left(  x,t\right)  =\min_{\hat{z}}\left\{  tL\left(  \frac{x-\hat{z}}%
{t}\right)  +g\left(  \hat{z}\right)  \right\}
\end{equation}
so substituting the $z$ defined above in (\ref{L11}) yields the inequality%
\begin{equation}
u\left(  x,t\right)  \leq tL\left(  \frac{x-z}{t}\right)  +g\left(  z\right)
\end{equation}
and now using (\ref{L12}) yields%
\begin{equation}
u\left(  x,t\right)  \leq\left(  t-s\right)  L\left(  \frac{x-y}{t-s}\right)
+sL\left(  \frac{y-z}{s}\right)  +g\left(  z\right)  . \label{L13}%
\end{equation}
Now note that the last two terms, by (\ref{L11}) are $u\left(  y,s\right)  .$
This yields%
\begin{equation}
u\left(  x,t\right)  \leq\left(  t-s\right)  L\left(  \frac{x-y}{t-s}\right)
+u\left(  y,s\right)  . \label{L13'}%
\end{equation}

Note that $y$ \ has been arbitrary. Now we take the minimum over all
$y\in\mathbb{R}$. We note that there are values of $y$ for which the right
hand side of (\ref{L13'}) is infinite, but given $0<s<t$ and $x\in\mathbb{R}$,
there will be some $y\in\mathbb{R}$ such that $\left(  x-y\right)  /\left(
t-s\right)  $ falls in the finite range of $L.$ Thus, in taking the minimum,
the values for which it is infinite are irrelevant, and we have%
\begin{equation}
u\left(  x,t\right)  \leq\min_{y\in\mathbb{R}}\left\{  \left(  t-s\right)
L\left(  \frac{x-y}{t-s}\right)  +u\left(  y,s\right)  \right\}  . \label{L14}%
\end{equation}

B. Next, we know again that there exists $w\in\mathbb{R}$ (depending on $x$
and $t$ that we regard as fixed) such that
\begin{equation}
u\left(  x,t\right)  =tL\left(  \frac{x-w}{t}\right)  +g\left(  w\right)  .
\label{L15}%
\end{equation}
We choose $y:=\frac{s}{t}x+\left(  1-\frac{s}{t}\right)  w,$ which implies%
\begin{equation}
\frac{x-y}{t-s}=\frac{x-w}{t}=\frac{y-w}{s}. \label{L16}%
\end{equation}

We know that $w$ is the minimizer (and of course, $g$ is finite in the domain
$\left[  m_{1},m_{N}\right]  $) so that $L\left(  \left(  x-w\right)
/t\right)  $ is finite. Thus, by the identity (\ref{L16}) above, so are
$L\left(  \left(  x-w\right)  /t\right)  $ and $L\left(  \left(  y-w\right)
/s\right)  .$ Thus, using the basic definition of $u\left(  y,s\right)  $ in
the equality, one has
\begin{align}
\left(  t-s\right)  L\left(  \frac{x-y}{t-s}\right)  +u\left(  y,s\right)   &
=\left(  t-s\right)  L\left(  \frac{x-y}{t-s}\right)  +\min_{\hat{z}}\left\{
sL\left(  \frac{y-\hat{z}}{t}\right)  +g\left(  \hat{z}\right)  \right\}
\nonumber\\
&  \leq\left(  t-s\right)  L\left(  \frac{x-y}{t-s}\right)  +\left\{
sL\left(  \frac{y-w}{t}\right)  +g\left(  w\right)  \right\}  \label{L17}%
\end{align}
where the inequality is obtained simply by substituting a particular value for
$\hat{z},$ namely the $w$ that we defined above in (\ref{L15}).

We can use (\ref{L16}) in order to re-write the arguments of $L$ in equivalent
forms. By the equality and the fact that $L\left(  \left(  x-w\right)
/t\right)  $ is finite, so are $L\left(  \left(  x-y\right)  /\left(
t-s\right)  \right)  $ and $L\left(  \left(  y-w\right)  /s\right)  .$ Hence,
replacing the two expressions involving $L$ on the RHS of (\ref{L12}) yields%
\begin{align}
\left(  t-s\right)  L\left(  \frac{x-y}{t-s}\right)  +u\left(  y,s\right)   &
\leq\left(  t-s\right)  L\left(  \frac{x-w}{t}\right)  +\left\{  sL\left(
\frac{x-w}{t}\right)  +g\left(  w\right)  \right\} \nonumber\\
&  =tL\left(  \frac{x-w}{t}\right)  +g\left(  w\right)  =u\left(  x,t\right)
\label{L18}%
\end{align}
where the last identity follows from the expression (\ref{L15}) that defines
$w.$ Thus (\ref{L18}) gives us an identity for a particular $y$ that we
defined, namely%
\begin{equation}
\left(  t-s\right)  L\left(  \frac{x-y}{t-s}\right)  +u\left(  y,s\right)
\leq u\left(  x,t\right)  . \label{L19}%
\end{equation}
If we replace $y$ by the minimum over all $\hat{y}$ we obtain the inequality%
\begin{equation}
\min_{\hat{y}\in\mathbb{R}}\left\{  \left(  t-s\right)  L\left(  \frac
{x-y}{t-s}\right)  +u\left(  y,s\right)  \right\}  \leq u\left(  x,t\right)  .
\label{L110}%
\end{equation}
Combining (\ref{L14}) and (\ref{L110}) proves Lemma \ref{Lem 1}.
\end{proof}

\bigskip

\begin{lemma}
\label{Lem B} If $L$ and $g$ are Lipschitz continuous, one has $u\left(
x,0\right)  =g\left(  x\right)  .$
\end{lemma}

Note that this is the analog of Lemma 2 - proof in part 2 of Evans. Part 2 is
essentially the same; one needs only pay attention to finiteness of the terms.

\begin{proof}
[Proof of Lemma \ref{Lem B}]Since $0\in\left[  m_{1},m_{N+1}\right]  $, the
interval on which $L$ is finite, one has
\begin{equation}
u\left(  x,t\right)  =\min_{y\in\mathbb{R}}\left\{  tL\left(  \frac{x-y}%
{t}\right)  +g\left(  y\right)  \right\}  \leq tL\left(  0\right)  +g\left(
x\right)  , \label{B1}%
\end{equation}
upon choosing $y=x.$ Also,%
\begin{align}
u\left(  x,t\right)   &  =\min_{y\in\mathbb{R}}\left\{  tL\left(  \frac
{x-y}{t}\right)  +g\left(  y\right)  \right\} \nonumber\\
&  \geq\min_{y\in\mathbb{R}}\left\{  tL\left(  \frac{x-y}{t}\right)  +g\left(
x\right)  -Lip\left(  g\right)  \cdot\left\vert x-y\right\vert \right\}
\nonumber\\
&  =g\left(  x\right)  +\min_{y\in\mathbb{R}}\left\{  tL\left(  \frac{x-y}%
{t}\right)  -Lip\left(  g\right)  \cdot\left\vert x-y\right\vert \right\}
\nonumber\\
&  =g\left(  x\right)  -t\max_{y\in\mathbb{R}}\left\{  Lip\left(  g\right)
\cdot\frac{\left\vert x-y\right\vert }{t}-L\left(  \frac{x-y}{t}\right)
\right\} \nonumber\\
&  =g\left(  x\right)  -t\max_{z\in\mathbb{R}}\left\{  \left\vert z\right\vert
Lip\left(  g\right)  -L\left(  z\right)  \right\}  . \label{B2}%
\end{align}
Note that $\max_{z\in\mathbb{R}}\left\{  \left\vert z\right\vert Lip\left(
g\right)  -L\left(  z\right)  \right\}  $ is a finite number since $-L\left(
z\right)  $ is bounded above and is $-\infty$ outside of the range $\left[
m_{1},m_{N}+1\right]  $. Thus we define%
\begin{equation}
C:=\max\left\{  \left\vert L\left(  0\right)  \right\vert ,\max_{z\in
\mathbb{R}}\left\{  \left\vert z\right\vert Lip\left(  g\right)  -L\left(
z\right)  \right\}  \right\}
\end{equation}
and combine (\ref{B1}) and (\ref{B2}) to write%
\begin{equation}
\left\vert u\left(  x,t\right)  -g\left(  x\right)  \right\vert \leq
Ct\ \text{\ for all }x\in\mathbb{R}\text{ and }t>0.
\end{equation}

\end{proof}

\bigskip

\begin{lemma}
\label{Lem C}(a) If $L$ and $g$ are Lipschitz, one has for any $x\in
\mathbb{R}$ and $0<\hat{t}<t$ the inequalities%
\begin{equation}
\left\vert u\left(  x,t\right)  -u\left(  x,\hat{t}\right)  \right\vert \leq
C\left\vert t-\hat{t}\right\vert
\end{equation}
\begin{equation}
C:=\max\left\{  \left\vert L\left(  0\right)  \right\vert ,\ \ \max
_{z\in\mathbb{R}}\left\{  \left\vert z\right\vert Lip\left(  g\right)
-L\left(  z\right)  \right\}  \right\}  .
\end{equation}
(b) Under the conditions of Lemmas \ref{Lem A} and \ref{Lem C} one has for
some $C$%
\begin{equation}
\left\vert u\left(  \hat{x},\hat{t}\right)  -u\left(  x,t\right)  \right\vert
\leq C\left\Vert \left(  \hat{x},\hat{t}\right)  -\left(  x,t\right)
\right\Vert _{2}%
\end{equation}
where $\left\Vert \cdot\right\Vert _{2}$ is the usual Euclidean norm.

(c) If $L$ and $g$ are Lipschitz continuous then $u:\mathbb{R}^{2}%
\rightarrow\mathbb{R}$ is differentiable on $\mathbb{R\times}\left(
0,\infty\right) a.e.$
\end{lemma}

\begin{proof}
[Proof of Lemma \ref{Lem C}](a) Let $x\in\mathbb{R}$ and $0<\hat{t}<t.$ By
Lemma \ref{Lem A} one has%
\begin{equation}
\left\vert u\left(  x,t\right)  -u\left(  \hat{x},t\right)  \right\vert \leq
Lip\left(  g\right)  \left\vert x-\hat{x}\right\vert \ .
\end{equation}
From Lemma \ref{Lem 1}, we have for $0\leq s=\hat{t}<t,$ the inequality%
\begin{align}
u\left(  x,t\right)   &  =\min_{y}\left\{  \left(  t-s\right)  L\left(
\frac{x-y}{t-s}\right)  +u\left(  y,s\right)  \right\} \nonumber\\
&  \geq\min_{y}\left\{  \left(  t-s\right)  L\left(  \frac{x-y}{t-s}\right)
+u\left(  x,s\right)  -Lip\left(  u\right)  \cdot\left\vert x-y\right\vert
\right\} \nonumber\\
&  =u\left(  x,s\right)  +\min_{y}\left\{  \left(  t-s\right)  L\left(
\frac{x-y}{t-s}\right)  -Lip\left(  g\right)  \cdot\left\vert x-y\right\vert
\right\} \nonumber\\
&  =u\left(  x,s\right)  +\left(  t-s\right)  \min_{y}\left\{  L\left(
\frac{x-y}{t-s}\right)  -Lip\left(  g\right)  \cdot\frac{\left\vert
x-y\right\vert }{t-s}\right\} \nonumber\\
&  =u\left(  x,s\right)  +\left(  t-s\right)  \min_{z}\left\{  L\left(
z\right)  -Lip\left(  g\right)  \cdot\left\vert z\right\vert \right\}
\label{C1}%
\end{align}
where we have just defined $z:=\left\vert x-y\right\vert /\left(  t-s\right)
.$ We can also write this as%
\begin{equation}
u\left(  x,t\right)  \geq u\left(  x,s\right)  -\left(  t-s\right)  \max
_{z}\left\{  Lip\left(  g\right)  \cdot\left\vert z\right\vert -L\left(
z\right)  \right\}  .
\end{equation}
Using $C_{1}:=\max_{z}\left\{  Lip\left(  g\right)  \cdot\left\vert
z\right\vert -L\left(  z\right)  \right\}  $ which, as discussed above, is
clearly finite, one has then%
\begin{equation}
u\left(  x,t\right)  -u\left(  x,s\right)  \geq-C_{1}\left(  t-s\right)  .
\end{equation}
The other direction in the inequality follows from Lemma \ref{Lem 1}
directly. Substituting $x$ in place of the $y$ in the minimizer, we have%
\begin{align}
u\left(  x,t\right)   &  =\min_{y}\left\{  \left(  t-s\right)  L\left(
\frac{x-y}{t-s}\right)  +u\left(  y,s\right)  \right\} \nonumber\\
&  \leq\left(  t-s\right)  L\left(  \frac{x-x}{t-s}\right)  +u\left(
x,s\right) \nonumber\\
&  =\left(  t-s\right)  L\left(  0\right)  +u\left(  x,s\right)  \label{C2}%
\end{align}
yielding the inequality%
\begin{equation}
u\left(  x,t\right)  -u\left(  x,s\right)  \leq\left(  t-s\right)  L\left(
0\right)  . \label{C3}%
\end{equation}

Combining (\ref{C2}) and (\ref{C3}) yields the Lipschitz inequality in $t,$
namely,%
\begin{equation}
\left\vert u\left(  x,t\right)  -u\left(  x,s\right)  \right\vert \leq
C\left\vert t-s\right\vert \ . \label{C4}%
\end{equation}

(b) This follows from the triangle inequality, Lemma \ref{Lem A} and part (a).

(c) This follows from Rademacher's theorem and part (b).
\end{proof}

Analogous to theorems in Section 3.3, \cite{EV}, we have the following three
theorems. The key idea here is that our versions allow one to deal with the
introduction of potentially infinite values of the Legendre transform of the
flux function.

\begin{theorem}
\label{Thm 5}Let $x\in\mathbb{R}$ and $t>0$. Let $u$ be defined by the
Hopf-Lax formula and differentiable at $\left(  x,t\right)  \in\mathbb{R\times
}\left(  0,\infty\right)  .$ Then
\[
u_{t}\left(  x,t\right)  +H\left(  u_{x}\left(  x,t\right)  \right)  =0.
\]

\end{theorem}

\begin{proof}
[Proof of Theorem \ref{Thm 5}]A. Fix $q\in\left[  m_{1},m_{N+1}\right]  $ and
$h>0.$ By Lemma \ref{Lem 1}, we have%
\begin{equation}
u\left(  x+hq,t+h\right)  =\min_{y\in\mathbb{R}}\left\{  hL\left(
\frac{x+hq-y}{h}\right)  +u\left(  y,t\right)  \right\}  .
\end{equation}
Once again since there are some finite values over which we are taking the
minimum, the expression is well-defined. Upon setting $y$ as $x,$ we can only
obtain a larger quantity on the RHS, yielding%
\begin{equation}
u\left(  x+hq,t+h\right)  \leq hL\left(  q\right)  +u\left(  x,t\right)  .
\end{equation}
Hence, for $q\in\left[  m_{1},m_{N+1}\right]  $, we have the inequality%
\begin{equation}
\frac{u\left(  x+hq,t+h\right)  -u\left(  x,t\right)  }{h}\leq L\left(
q\right)  . \label{Thm51}%
\end{equation}

Since we are assuming that $u$ is differentiable at $\left(  x,t\right)  $ we
have the existence of the limit of the LHS of (\ref{Thm51}) thereby yielding
\begin{align}
\partial_{t}u\left(  x,t\right)  +qDu\left(  x,t\right)   &  \leq L\left(
q\right)  ,\ i.e.,\nonumber\\
\partial_{t}u\left(  x,t\right)  +qDu\left(  x,t\right)  -L\left(  q\right)
&  \leq0. \label{Thm52}%
\end{align}

We now use the equality, $\sup_{q\in\mathbb{R}}\left\{  qw-L\left(  q\right)
\right\}  =:H\left(  w\right)  $, writing%
\begin{equation}
H\left(  Du\left(  x,t\right)  \right)  =\sup_{q\in\mathbb{R}}\left\{
qDu-L\left(  q\right)  \right\}  .
\end{equation}

Note that the values of $q$ for which $L\left(  q\right)  =\infty$ are clearly
not candidates for the $\sup_{p}$ since $-L\left(  q\right)  =-\infty$. Hence,
we can take the sup over all $q$ that satisfy (\ref{Thm52})$,$ (which is
equivalent to taking the sup over $q\in\left[  m_{1},m_{N+1}\right]  $) to
obtain%
\begin{align}
\partial_{t}u\left(  x,t\right)  +\sup_{q\in\mathbb{R}}\left\{  \partial
_{t}u\left(  x,t\right)  +qDu\left(  x,t\right)  -L\left(  q\right)  \right\}
&  \leq0,\ \ or\nonumber\\
\partial_{t}u\left(  x,t\right)  +H\left(  Du\left(  x,t\right)  \right)   &
\leq0. \label{Thm53}%
\end{align}

B. Now use the definition%
\begin{equation}
u\left(  x,t\right)  =\min_{y\in\mathbb{R}}\left\{  tL\left(  \frac{x-y}%
{t}\right)  +g\left(  y\right)  \right\}  .
\end{equation}
Since $L$ and $g$ are continuous, the minimizer exists and for some
$z\in\mathbb{R}$ depending on $\left(  x,t\right)  $ we have
\begin{equation}
u\left(  x,t\right)  =tL\left(  \frac{x-z}{t}\right)  +g\left(  z\right)  .
\label{Thm54}%
\end{equation}
Define $s:=t-h,$ $y:=\frac{s}{t}x+\left(  1-\frac{s}{t}\right)  z$ so
$\frac{x-z}{t}=\frac{y-z}{s}.$ Then we can write, using the definition of
$u\left(  y,s\right)  $,%
\begin{equation}
u\left(  x,t\right)  -u\left(  y,s\right)  =tL\left(  \frac{x-z}{t}\right)
+g\left(  z\right)  -\min_{\hat{y}}\left\{  sL\left(  \frac{y-\hat{y}}%
{s}\right)  +g\left(  \hat{y}\right)  \right\}  .
\end{equation}
By substituting $z$ (defined by (\ref{Thm53})) in place of $\hat{y}$ in this
expression, we subtract out at least as much and obtain the inequality%
\begin{align}
u\left(  x,t\right)  -u\left(  y,s\right)   &  \geq tL\left(  \frac{x-z}%
{t}\right)  +g\left(  z\right)  -\left\{  sL\left(  \frac{y-z}{s}\right)
+g\left(  z\right)  \right\} \nonumber\\
&  =\left(  t-s\right)  L\left(  \frac{x-z}{t}\right)  \label{Thm55}%
\end{align}
by virtue of the equality $\frac{x-z}{t}=\frac{y-z}{s}.$ Note that by
definition, $L\left(  \frac{x-z}{t}\right)  =L\left(  \frac{y-z}{s}\right)
<\infty,$ so there is no divergence problem there. Now replace $y$ with its
definition above, and use $t-s=h$ to write (\ref{Thm55}) as%
\begin{equation}
\frac{u\left(  x,t\right)  -u\left(  x+\frac{h}{t}\left(  z-x\right)
,t-h\right)  }{h}\geq L\left(  \frac{x-z}{t}\right)  . \label{Thm56}%
\end{equation}
Since we are assuming that the derivative exists at $\left(  x,t\right)  $ we
can take the limit as $h\rightarrow0+$ and obtain%
\begin{equation}
\frac{x-z}{t}Du\left(  x,t\right)  +\partial_{t}u\left(  x,t\right)  \geq
L\left(  \frac{x-z}{t}\right)  \label{Thm57}%
\end{equation}

We use the definition of $H$ again, and write%
\begin{align}
u_{t}\left(  x,t\right)  +H\left(  Du\left(  x,t\right)  \right)   &
=u_{t}\left(  x,t\right)  +\max_{q\in\mathbb{R}}\left\{  qDu-L\left(
q\right)  \right\} \nonumber\\
&  \geq u_{t}\left(  x,t\right)  +\frac{x-z}{t}Du\left(  x,t\right)  -L\left(
\frac{x-z}{t}\right)  \label{Thm58}%
\end{align}
where we have chosen one value of $q,$ namely $\left(  x-z\right)  /t$ to
obtain the inequality. Note, again, that from the original definition in
(\ref{Thm54})$,$ $L\left(  \frac{x-z}{t}\right)  $ must be finite. Hence, the
RHS\ of (\ref{Thm58}) is well-defined. Combining (\ref{Thm57}) and
(\ref{Thm58}) yields the inequality%
\begin{equation}
\partial_{t}u\left(  x,t\right)  +H\left(  Du\left(  x,t\right)  \right)
\geq0. \label{Thm59}%
\end{equation}

Combining (\ref{Thm59}) with (\ref{Thm54})$,$ we obtain the result that $u$
satisfies the Hamilton-Jacobi equation at $\left(  x,t\right)  $.
\end{proof}

\begin{theorem}
\label{Thm 6}The function $u\left(  x,t\right)  $ defined by the Lax-Oleinik
formula is Lipschitz continuous, differentiable a.e. in $\mathbb{R\times
}\left(  0,\infty\right)  $ and solves the initial value problem%
\begin{align*}
u_{t}+H\left(  u_{x}\right)   &  =0\ \ \ a.e.\ in\ \mathbb{R\times}\left(
0,\infty\right) \\
u\left(  x,0\right)   &  =g\left(  x\right)  \ \ for\ all\ x\in\mathbb{R}%
\text{ .}%
\end{align*}

\end{theorem}

\bigskip

\begin{definition}
We say that $w\in L^{\infty}\left(  \mathbb{R\times}\left(  0,\infty\right)
\right)  $ is an\textbf{\ }\textit{integral solution}\textbf{ }of
\begin{align*}
w_{t}+H\left(  w\right)  _{x}  &  =0\ in\ \mathbb{R\times}\left(
0,\infty\right) \\
w\left(  x,0\right)   &  =h\left(  x\right)  \ for\ all\ x\in\mathbb{R}\text{
}%
\end{align*}
if for all test functions $\phi:\mathbb{R\times}[0,\infty)\rightarrow
\mathbb{R}$ (i.e., $\phi$ that are smooth and have compact support) one has
the identity%
\[
\int_{0}^{\infty}\int_{-\infty}^{\infty}w\phi_{t}dxdt+\int_{-\infty}^{\infty
}hx\phi dx|_{t=0}+\int_{0}^{\infty}\int_{-\infty}^{\infty}H\left(  w\right)
\phi_{x}dxdt=0.
\]

\end{definition}

\begin{theorem}
\label{Thm 2}Under the assumptions that $g$ is Lipschitz and that $H$ is
polygonal and convex (as described above), the function $w\left(  x,t\right)
:=\partial_{x}u\left(  x,t\right)  $ where $u$ is the Hopf-Lax function is an
integral solution of the initial value problem for the conservation law above.
\end{theorem}

This is the analog of Theorem 2, Section 3.4 of \cite{EV}, but the statement
of the theorem there is somewhat different.

\begin{proof}
[Proof of Theorem \ref{Thm 2}]From Theorem \ref{Thm 6} above we know that $u$
is Lipschitz continuous, differentiable a.e. in $\mathbb{R\times}\left(
0,\infty\right)  $ and solves the Hamilton-Jacobi initial value problem
subject to initial condition $h\left(  x\right)  $ where $g\left(  x\right)
=\int_{0}^{x}h\left(  z\right)  dz$. We multiply $u_{t}+H\left(  u_{x}\right)
=0$ with a test function $\phi_{x}$ and integrate over $\int_{0}^{\infty}%
\int_{-\infty}^{\infty}...dxdt$. Upon integrating by parts one obtains the
relation above in the definition of integral solution. The integration by
parts operations are justified by the fact that $u\left(  x,t\right)  $ is
Lipschitz in both $x$ and $t$. Also,%
\[
w\left(  x,0\right)  =\partial_{x}u\left(  x,0\right)  =g^{\prime}\left(
x\right)  .
\]

\end{proof}

Notably, we have used the largest of the minimizers rather than the least to
improve on the result of \cite{EV} by requiring only convexity instead of
uniform convexity of the flux function.

\section{Proof That Solution is BV and Greatest Minimizer is Non-decreasing in
$x$}

\begin{theorem}
\label{Thm 4.1}Suppose $H$ is polygonal (with finitely many break points),
convex, $H\left(  0\right)  =0$, and $g$ is differentiable. Then for any
$\left(  x,t\right)  $ there exists $y^{\ast}\left(  x,t\right)  $ that is
defined as the greatest minimizer, i.e.,%
\begin{equation}
\min_{y\in\mathbb{R}}\left\{  tL\left(  \frac{x-y}{t}\right)  +g\left(
y\right)  \right\}  =\left\{  tL\left(  \frac{x-y^{\ast}\left(  x,t\right)
}{t}\right)  +g\left(  y^{\ast}\left(  x,t\right)  \right)  \right\}
\end{equation}
and any other number $\hat{y}$ that minimizes the left hand side satisfies
$\hat{y}\leq y^{\ast}.$
\end{theorem}

\begin{remark}
By using the largest minimizer instead of the least as in the classical
theorems, we obtain a particular inequality below that is a consequence of
convexity rather than from the stricter assumption of uniform convexity.
\end{remark}

\begin{remark}
(\textbf{Minimizers)} We first illustrate the key idea. The minimizers must
either be on the vertices of $L$ or on the locally differentiable part of $L$.
We suppress $t$ and suppose $x=0$. For fixed $x,t$ in order to have a
non-isolated set of minimizers of $f\left(  y;x,t\right)  +g\left(  y\right)
$, they need to be on the differentiable part of $f$ (i.e. non-vertex). This
latter case means that on some interval, e.g., $\left[  a,b\right]  $ one has
$f\left(  y;0,t\right)  +g\left(  y\right)  =0$ (note that we can always shift
up or down, so we can adjust the constant to $0$). On this stretch of $y$ we
can write%
\[
f\left(  y;0\right)  =my\ \ and\ -g\left(  y\right)  =my
\]
Thus all $y\in\left[  a,b\right]  $ are minimizers. If we increase $x$
slightly we obtain
\[
f\left(  y;x\right)  =m\left(  y-x\right)  \ \ and\ -g\left(  y\right)  =my
\]
so that
\[
f\left(  y;x\right)  +g\left(  y\right)  =m\left(  y-x\right)  -my=-mx.
\]
This means that the minimum is less (if $x>0$) but again, all $y$ are
minimizers. In computing the derivative%
\[
\partial_{x}\min_{y}\left\{  f\left(  y;x\right)  +g\left(  y\right)
\right\}
\]
we see that we can use any $y$ and we will obtain the same result. I.e., if
$\hat{y}$ is any minimizer, then we have (as one can see graphically, or from
computation), where $g$ is differentiable,
\[
\partial_{x}\min_{y}\left\{  f\left(  y;x\right)  +g\left(  y\right)
\right\}  =-\partial_{y}f\left(  \hat{y};x\right)  =g^{\prime}\left(  \hat
{y}\right)  .
\]
Hence, we can take for example the largest of these $\hat{y},$ or $\sup\hat
{y}$ since the derivatives are constant (and, $f$ and $g$ are $C^{1}$ so we
can use continuity).
\end{remark}

\begin{center}

\end{center}

\begin{proof}
[Proof of Theorem \ref{Thm 4.1}]Note first that by definition of the Legendre
transform, $L\left(  q\right)  $, one has that $L\left(  q\right)  <\infty$ if
and only if $q\in\left[  m_{1},m_{N+1}\right]  .$ Also, the interval is
closed, since one has for some $c\in\mathbb{R}$,
\begin{equation}
L\left(  m_{N+1}\right)  =\sup_{q}\left\{  m_{N+1}q-H\left(  q\right)
\right\}  =c
\end{equation}
for some constant $c$, and similarly at the $m_{1}$ endpoint. Hence, if we
define, for fixed $x$ and $t,$%
\begin{equation}
v\left(  y\right)  :=tL\left(  \frac{x-y}{t}\right)  +g\left(  y\right)
\end{equation}
and note that $v$ is continuous, since, by assumption $g$ and $L$ are also continuous.

A continuous function on a closed, bounded interval attains at least one
minimizer, which we will call $y_{1}\in\left[  x-m_{N+1}t,x-m_{1}t\right]  $
and denote $v\left(  y_{1}\right)  =:b.$ Note that $v$ is infinite outside of
this interval. Now for any index set $A$, let $\left\{  y_{\alpha}\right\}
_{\alpha\in A}$ be the set of points such that $y_{\alpha}\in\left[
x-m_{N+1}t,x-m_{1}t\right]  $ and $v\left(  y_{\alpha}\right)  =v\left(
y_{1}\right)  .$ Let $y^{\ast}:=\sup\left\{  y_{\alpha}:\alpha\in A\right\}  $
which exists since it is a bounded set of reals. Thus there is a sequence
$\left\{  y_{j}\right\}  _{j=1}^{\infty}$ such that $\lim_{j\rightarrow\infty
}y_{j}=y^{\ast}$. By continuity of $v$ one has then that
\begin{equation}
v\left(  y^{\ast}\right)  =\lim_{j\rightarrow\infty}v\left(  y_{j}\right)
=\lim_{j\rightarrow\infty}b=b.
\end{equation}
Thus, $y^{\ast}$ is a minimizer, there is no minimizer that is greater than
$y^{\ast},$ and it is unique by definition of supremum.
\end{proof}

We write $y^{\ast}\left(  x,t\right)  $ as the largest minimizer for a given
$x$ and $t$. We define $y_{1}^{\ast}:=y^{\ast}\left(  x_{1},t\right)  $ and
$y_{2}^{\ast}=y^{\ast}\left(  x_{2},t\right)  .$\textbf{\ }We have the
following, analogous to \cite{EV} but using convexity that may not be strict.

\begin{lemma}
\label{Lem 4.1}If $x_{1}<x_{2}$, then%
\[
tL\left(  \frac{x_{2}-y_{1}^{\ast}}{t}\right)  +g\left(  y_{1}^{\ast}\right)
\leq tL\left(  \frac{x_{2}-y}{t}\right)  +g\left(  y\right)
\ \ \ \ if\ \ \ y<y_{1}^{\ast}.
\]

\end{lemma}

\begin{proof}
[Proof of Lemma \ref{Lem 4.1}]For any $\lambda\in\left[  0,1\right]  $ and
$r,s\in\mathbb{R}$ we have from convexity
\begin{equation}
L\left(  \lambda r+\left(  1-\lambda\right)  s\right)  \leq\lambda L\left(
r\right)  +\left(  1-\lambda\right)  L\left(  s\right)  .\
\end{equation}
Now let $x_{1}<x_{2}$ and for $y\in y_{1}^{*}$ define $\lambda$ by%
\begin{equation}
\lambda:=\frac{y_{1}^{\ast}-y}{x_{2}-x_{1}+y_{1}^{\ast}-y}\ .
\end{equation}
By assumption, $x_{1}<x_{2}$ and $y<y_{1}^{\ast}$ so that $\lambda\in\left(
0,1\right)  $. Let $r:=\left(  x_{1}-y_{1}^{\ast}\right)  /t$ and $s:=\left(
x_{2}-y\right)  /t.$ A computation shows
\begin{equation}
\lambda r+\left(  1-\lambda\right)  s=\left(  x_{2}-y\right)  /t\ .
\end{equation}
This yields%
\begin{equation}
L\left(  \frac{x_{2}-y_{1}^{\ast}}{t}\right)  \leq\lambda L\left(  \frac
{x_{1}-y_{1}^{\ast}}{t}\right)  +\left(  1-\lambda\right)  L\left(
\frac{x_{2}-y}{t}\right)  . \label{Lem411}%
\end{equation}
Next, we interchange the roles of $r$ and $s$, i.e, $\hat{r}:=\left(
x_{2}-y\right)  /t$ and $\hat{s}:=\left(  x_{1}-y_{1}^{\ast}\right)  /t$ and
note that
\begin{equation}
\lambda\hat{r}+\left(  1-\lambda\right)  \hat{s}=\left(  x_{1}-y\right)  /t.
\end{equation}
Thus, convexity implies%
\begin{equation}
L\left(  \frac{x_{1}-y}{t}\right)  \leq\left(  1-\lambda\right)  L\left(
\frac{x_{1}-y_{1}^{\ast}}{t}\right)  +\lambda L\left(  \frac{x_{2}-y}%
{t}\right)  . \label{Lem412}%
\end{equation}
Adding (\ref{Lem411}) and (\ref{Lem412}) yields,%
\begin{equation}
L\left(  \frac{x_{2}-y_{1}^{\ast}}{t}\right)  +L\left(  \frac{x_{1}-y}%
{t}\right)  \leq L\left(  \frac{x_{1}-y_{1}^{\ast}}{t}\right)  +L\left(
\frac{x_{2}-y}{t}\right)  . \label{Lem413}%
\end{equation}
By definition of $y_{1}^{\ast}:=y^{\ast}\left(  x_{1},t\right)  $ as a
minimizer for $x_{1}$ (with $t$ still fixed) we have%
\begin{equation}
tL\left(  \frac{x_{1}-y_{1}^{\ast}}{t}\right)  +g\left(  y_{1}^{\ast}\right)
\leq tL\left(  \frac{x_{1}-y}{t}\right)  +g\left(  y\right)  \ .
\label{Lem414}%
\end{equation}
Upon multiplying (\ref{Lem413}) by $t$ and adding to (\ref{Lem414}) we obtain,
as two of the $L$ terms cancel,%
\begin{equation}
tL\left(  \frac{x_{2}-y_{1}^{\ast}}{t}\right)  +g\left(  y_{1}^{\ast}\right)
\leq tL\left(  \frac{x_{2}-y}{t}\right)  +g\left(  y\right)  , \label{Lem415}%
\end{equation}
provided, still, that $y<y_{1}^{\ast}$ . Hence, if $y_{2}^{\ast}$ is the
largest minimizer for $x_{2},$ it must be greater than or equal to
$y_{1}^{\ast}.$ I.e., any value $y<y_{1}^{\ast}$ satisfies (\ref{Lem415}) so
it could not be a larger minimizer than $y_{1}^{\ast}.$
\end{proof}

An immediate consequence of this result is the following:

\begin{theorem}
\label{Thm 4.2}For each fixed, $t,$ as a function of $x,$ $y^{\ast}\left(
x,t\right)  $ is non-decreasing and is equal a.e. in $x$ to a differentiable function.
\end{theorem}

\begin{proof}
[Proof of Theorem \ref{Thm 4.2}]1. From previous calculations we know that $L$
is polygonal convex and is finite only in the closed interval $\left[
m_{1},m_{N+1}\right]  .$ Also, we have $L\geq0$. Now we apply Lemma
\ref{Lem 4.1} as follows. By definition of $y_{2}^{\ast}$ we have
\begin{equation}
\min_{y\in\mathbb{R}}\left\{  tL\left(  \frac{x_{2}-y}{t}\right)  +g\left(
y\right)  \right\}  =\left\{  tL\left(  \frac{x_{2}-y_{2}^{\ast}\left(
x,t\right)  }{t}\right)  +g\left(  y_{2}^{\ast}\left(  x,t\right)  \right)
\right\}  .
\end{equation}
By Proposition \ref{Prop1}, we have that for any $y<y_{1}^{\ast}$, the
expression$\left\{  tL\left(  \frac{x_{2}-y}{t}\right)  +g\left(  y\right)
\right\}  $ is already equal to or greater than $tL\left(  \frac{x_{2}%
-y_{1}^{\ast}}{t}\right)  +g\left(  y_{1}^{\ast}\right)  ,$ so there cannot be
a minimizer that is less than $y_{1}^{\ast}\ .$ Hence, we conclude
$y_{2}^{\ast}\geq y_{1}^{\ast}$ .

This means that with $y^{\ast}\left(  x,t\right)  $ defined as the largest
value that minimizes $tL\left(  \frac{x_{2}-y}{t}\right)  +g\left(  y\right)
$, i.e.%
\[
y^{\ast}\left(  x,t\right)  =\arg^{+}\min\left\{  tL\left(  \frac{x_{2}-y}%
{t}\right)  +g\left(  y\right)  \right\}
\]
we have that the function $y^{\ast}$ is a non-decreasing function of $x$ for
any fixed $t>0$. This implies that it is continuous except for countably many
values of $x.$ Also, for each $t>0,$ one has that $y^{\ast}\left(  x,t\right)
$ is equal a.e. to a function $\hat{y}\left(  x\right)  $ such that $\hat{y}$
is differentiable in $x$ and one has $\hat{y}\left(  x\right)  =\int_{0}%
^{x}\hat{y}^{\prime}\left(  s\right)  ds+z\left(  x\right)  $ where $z$ is
non-decreasing and $z'=0$ except on a set of measure zero (\cite{RU}, p. 157).
\end{proof}

\begin{lemma}
\label{Lem4.2}Let $g$ be differentiable,
\begin{align*}
v\left(  x,t\right)   &  :=tL\left(  \frac{x-y}{t}\right)  +g\left(  y\right)
,\\
w\left(  x,t\right)   &  :=\partial_{x}v\left(  x,t\right)  =\partial_{x}%
\min_{y\in\mathbb{R}}\left\{  tL\left(  \frac{x-y}{t}\right)  +g\left(
y\right)  \right\}  .
\end{align*}
Suppose that $\hat{y}\left(  x,t\right)  $ is the\textbf{\ unique} minimizer
of $v\left(  x,t\right)  $ and that $L\left(  \frac{x-y}{t}\right)  $ is
\textbf{differentiable }at $\hat{y}.$ Then%
\[
w\left(  x,t\right)  =g^{\prime}\left(  \hat{y}\left(  x,t\right)  \right)
=L\left(  \frac{x-\hat{y}}{t}\right)  .
\]

\end{lemma}

\begin{remark}
When we take the derivative of the minimum, note that the uniqueness of the
minimizer is the key issue. If there is more than one minimizer, as we vary
$y$ in order to take the derivative, one of the minimizers may become
irrelevant if the other minimum moves lower. This issue will be taken up in
the subsequent theorem.
\end{remark}

\begin{proof}
[Proof of Lemma \ref{Lem4.2}]Suppose that $x$ and $t$ are fixed and that
$\hat{y}\left(  x,t\right)  $ is the unique minimizer of $v\left(  x,t\right)
.$ Since $L,$ $g$ are differentiable, and $\hat{y}\left(  x,t\right)  $ is
also differentiable (since $\hat{y}$ is the only minimizer we can apply the
previous result on the greatest minimizer), we have the calculations:%
\[
0=\partial_{y}\left\{  tL\left(  \frac{x-y}{t}\right)  +g\left(  y\right)
\right\}  ,
\]
i.e.
\[
L^{\prime}\left(  \frac{x-\hat{y}\left(  x,t\right)  }{t}\right)  =g^{\prime
}\left(  \hat{y}\left(  x,t\right)  \right)  .
\]
Note that the minimum of $v$ will not occur at the minimum of $g$ unless $L$
has slope zero. We have then
\begin{align*}
w\left(  x,t\right)   &  :=\partial_{x}\min_{y\in\mathbb{R}}\left\{  tL\left(
\frac{x-y}{t}\right)  +g\left(  y\right)  \right\} \\
&  =\partial_{x}\left\{  tL\left(  \frac{x-\hat{y}\left(  x,t\right)  }%
{t}\right)  +g\left(  \hat{y}\left(  x,t\right)  \right)  \right\} \\
&  =tL^{\prime}\left(  \frac{x-\hat{y}\left(  x,t\right)  }{t}\right)
\cdot\left(  \frac{-\partial_{x}\hat{y}\left(  x,t\right)  }{t}+1\right)
+g^{\prime}\left(  \hat{y}(x,t)\right)  \partial_{x}\hat{y}\left(  x,t\right)
.
\end{align*}
The previous identity implies cancellation of the first and third terms,
yielding%
\[
w\left(  x,t\right)  =g^{\prime}\left(  \hat{y}\left(  x,t\right)  \right)  .
\]

\end{proof}

\begin{lemma}
\label{Lem 4.3}Let \thinspace$x,t$ be fixed, and assume the same conditions on
$L$ and $g.$ If $\hat{y}\left(  x,t\right)  $ is the unique minimizer of
$v\left(  x,t\right)  $ and occurs at a \textbf{vertex} of $L\left(
\frac{x-y}{t}\right)  ,$ then%
\[
w\left(  x,t\right)  =g^{\prime}\left(  \hat{y}\left(  x,t\right)  \right)  .
\]

\end{lemma}

\begin{proof}
[Proof of Lemma \ref{Lem 4.3}]Note that $L\left(  z\right)  <\infty$ if and
only if $z\in\left[  m_{1},m_{N+1}\right]  $,. Since the vertical coordinate
of the vertex of $L$ remains constant as one increases $x,$ the change in the minimum
is equal to $g^{\prime}\left(  y\right)  $ at that point. I.e., one has%
\begin{align*}
w\left(  x,t\right)   &  =\partial_{x}\min_{y\in\mathbb{R}}\left\{  tL\left(
\frac{x-y}{t}\right)  +g\left(  y\right)  \right\} \\
&  =\partial_{x}\left\{  tL\left(  \frac{x-\hat{y}\left(  x,t\right)  }%
{t}\right)  +g\left(  \hat{y}\left(  x,t\right)  \right)  \right\} \\
&  =g^{\prime}\left(  \hat{y}\left(  x,t\right)  \right)  .
\end{align*}
At the endpoints, $y=x-m_{N+1}t$ and $y=x-m_{1}t$ the situation is the same,
since as one varies $x$, the value of $\mathcal{L}$ on one side has an
infinite slope (see Figure \ref{PWF}). Note that this argument does not depend
on $\hat{y}$ being differentiable.
\end{proof}

\begin{theorem}
\label{Thm 4.3} Let $g$ be differentiable and $L$ convex polygonal as above.
For fixed $t>0$ and a.e. $x$ one has%
\begin{equation}
w\left(  x,t\right)  =g^{\prime}\left(  y^{\ast}\left(  x,t\right)  \right)  .
\end{equation}

\end{theorem}

\begin{proof}
[Proof of Theorem \ref{Thm 4.3}]Since $g$ is differentiable (for all $x$) any
minimum of $\left\{  tL\left(  \frac{x-y}{t}\right)  +g\left(  y\right)
\right\}  $ must occur on a point, $\hat{y}\left(  x,t\right)  $ where $L$ has
a vertex (including the endpoints, see Figure \ref{TypesMinima}) or at a point
where $v\left(  y\right)  =0.$ There are two possible types of minimizers,
Type $\left(  A\right)  $ occurring at the vertices of $L$, and Type $\left(
B\right)  $ that occur at the differentiable (i.e., flat part of $L$). These
two types are illustrated in Figure \ref{TypesMinima}. From the lemmas above,
we know that when there is a single minimizer, $\hat{y},$ the conclusion
follows. \ Thus we consider the possibility of more than one minimizer,
$\hat{y}_{j}.$
\begin{figure}
[ptb]
\begin{center}
\includegraphics[width=\linewidth]{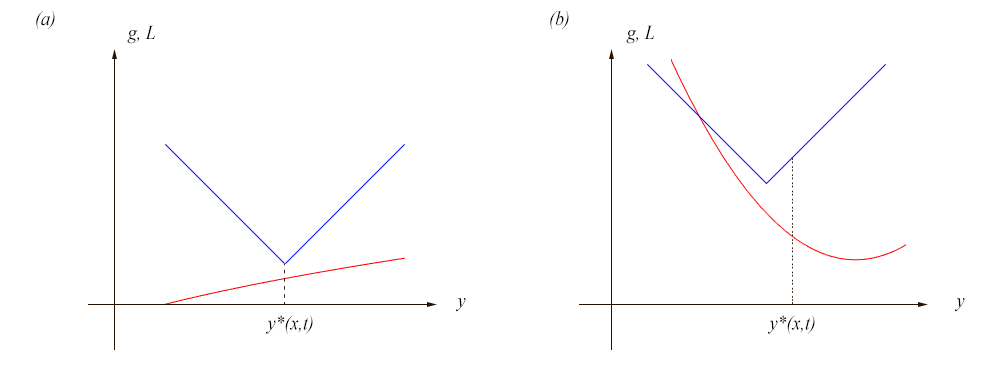}
\caption{The two types of minima that can occur for a piecewise flux function
and general initial data: (a) at a vertex of the Legendre transform $L\left(
q\right)  $ of the flux function $H\left(  q\right)  $; (b) at a point where
the Legendre transform $L\left(  q\right)  $ of the flux function $H\left(
q\right)  $ is locally differentiable.}
\label{TypesMinima}
\end{center}
\end{figure}
%EndExpansion

For a given $x,t$ it is clear that there can only be finitely many minimizers
of Type $\left(  A\right)  ,$ i.e., the number of vertices. Although there may
be infinitely many minimizers, $\hat{y}_{j}$ of the Type $\left(  B\right)  $
we know that $g^{\prime}\left(  \hat{y}\left(  x,t\right)  \right)
=L^{\prime}\left(  \frac{x-\hat{y}}{t}\right)  $ so that there are only
\textbf{finitely many values} of $g^{\prime}\left(  \hat{y}\left(  x,t\right)
\right)  $ regardless of the type of minimizer. For any $x,t$ we let $y^{\ast
}\left(  x,t\right)  $ be the largest of the minimizers, which is certainly
well-defined since there are only finitely many minimizers. From the earlier
theorem, we know that $y^{\ast}\left(  x,t\right)  $ is increasing in $x$ (for
fixed $t>0$) and differentiable for a.e. $x.$ In fact, if we focus on any
minimizer, $\hat{y}_{j}\left(  x,t\right)  $ we see that $\partial_{x}\hat
{y}_{j}\left(  x,t\right)  $ exists for either type of minimum. If it is Type
$\left(  A\right)  $ then as we vary $x,$ the vertex moves and the minimum
shifts along the curve of $g.$ Since $g$ is differentiable, the location of
the minimum varies smoothly in $x,$ so $\partial_{x}\hat{y}_{j}\left(
x,t\right)  $ exists. If it is Type $\left(  B\right)  $ then both $L$ and $g$
are differentiable, so it is certainly true that $\partial_{x}\hat{y}%
_{j}\left(  x,t\right)  $ exists.

For fixed $\left(  x,t\right)  $ and each of the finitely many values of
$g^{\prime}\left(  \hat{y}\left(  x,t\right)  \right)  $ we can determine the
minimum of $g^{\prime}\left(  \hat{y}\left(  x,t\right)  \right)  =:m$ (i.e.,
$m$ depending on $\left(  x,t\right)  $ ). First, they may correspond to
vertices. There are at most $M$ of those, since there can only be one
minimizer for each vertex. Then we have a class of minimizers for each segment
of $L\left(  \frac{x-y}{t}\right)  $, i.e. $M+1$ of those. We can take the
largest minimizer in each class, since $g^{\prime}$ will be the same in each
class. When we differentiate with respect to $x,$ we compare each of the $J$
(which is an integer between $1$ and $2M+1$ ) minimizers. It is the least of
these that will be relevant, since we are taking
\[
\partial_{x}\min_{y\in\mathbb{R}}\left\{  tL\left(  \frac{x-\hat{y}_{j}}%
{t}\right)  +g\left(  \hat{y}_{j}\right)  \right\}  .
\]
In other words, as we vary $x,$ we want to know how this minimum varies. Thus
a minimizer is irrelevant if $tL\left(  \frac{x-\hat{y}_{j}}{t}\right)
+g\left(  \hat{y}_{j}\right)  $ does not move down as much as another of the
$\hat{y}_{j}$ as $x$ increases. If $l\not =k$ and $g^{\prime}\left(  \hat
{y}_{k}\left(  x,t\right)  \right)  >g^{\prime}\left(  \hat{y}_{l}\left(
x,t\right)  \right)  $ then $\hat{y}_{k}\left(  x,t\right)  $ is irrelevant
for points beyond $x.$ On the other hand, if we have $g^{\prime}\left(
\hat{y}_{k}\left(  x,t\right)  \right)  =g^{\prime}\left(  \hat{y}_{l}\left(
x,t\right)  \right)  $ then we obtain the same change in the minimum, and we
can just take the larger of the two.

If we have minimizers that are of Type $\left(  B\right)  $ then it is the
furthest right segment that corresponds to the least $g^{\prime}$ since we
have the identity (see Lemma \ref{Lem4.2} above) $L^{\prime}\left(
\frac{x-\hat{y}\left(  x,t\right)  }{t}\right)  =g^{\prime}\left(  \hat
{y}\left(  x,t\right)  \right)  $.

In other words, either the $g^{\prime}$ values are identical on some interval,
in which case we have $w\left(  x,t\right)  =g^{\prime}\left(  \hat{y}%
_{j}\left(  x,t\right)  \right)  =g^{\prime}\left(  \hat{y}_{l}\left(
x,t\right)  \right)  $ for example, and we can take either minimizer and
obtain the same value for $w\left(  x,t\right)  ,$or one value is greater and
is thus irrelevant.

Alternatively, if the derivatives are different, then the smaller $g^{\prime}$
is the only one that is relevant. In either case we can take the largest value
of $\hat{y}$ and we have%
\[
w\left(  x,t\right)  =g^{\prime}\left(  y^{\ast}\left(  x,t\right)  \right)
.
\]

One remaining question is whether we have a largest minimizer $y^{\ast}\left(
x,t\right)  .$ The segments (i.e., lines of $L$) are on closed and bounded
intervals, the supremum of points $\hat{y}\left(  x,t\right)  $ exists, and
there is a sequence of points, $\tilde{y}_{m}$ that converges to this
supremum, $\tilde{y}$. Since $L$ and $g$ are continuous
\[
\lim_{m\rightarrow\infty}\left\{  tL\left(  \frac{x-\tilde{y}_{m}}{t}\right)
+g\left(  \tilde{y}_{m}\right)  \right\}  =tL\left(  \frac{x-\tilde{y}}%
{t}\right)  +g\left(  \tilde{y}\right)  .
\]
Thus, we must have that $\tilde{y}=y^{\ast}$ and is the greatest of the minimizers.
\end{proof}

\begin{theorem}
\label{Thm 4.4}Suppose that $g^{\prime}$ is BV. Then $w\left(  x,t\right)  $
is BV in $x$ for fixed $t.$
\end{theorem}

\begin{proof}
[Proof of \ref{Thm 4.4}]Since $g^{\prime}$ is BV it can be written as the
difference of two increasing functions, $h_{1}$ and $h_{2}.$ Then
$h_{i}\left(  y_{\ast}\left(  x,t\right)  \right)  $ are increasing (since
they are increasing functions of increasing functions) and hence $g^{\prime
}\left(  y_{\ast}\left(  x,t\right)  \right)  $ is BV.
\end{proof}

\section{Approximation Solutions of the Sharp Vertex Problem With the Smoothed
Version}

It is important to relate the solutions of the conservation law with the
polygonal flux $H$ to the solutions $w^{\varepsilon}$corresponding to the
smoothed and superlinearized flux function $\mathcal{H}^{\varepsilon}$. In
particular, $\mathcal{H}^{\varepsilon}$ is also uniformly convex, and the
hypotheses of the classical theorems are satisfied. Throughout this section we
will assume $g\in C^{1}$.

There are two basic parts to this section. First, we show that $H$ and
$\mathcal{H}^{\varepsilon}$ have Legendre transforms that are pointwise
separated by $C\varepsilon,$ i.e., $\left\vert L\left(  p\right)
-\mathcal{L}^{\varepsilon}\left(  p\right)  \right\vert \leq C\varepsilon$
also, and hence a similar identity for $f\left(  y;x,t\right)  =tL\left(
\frac{x-y}{t}\right)  $. We will also show that if there is a unique pair of
minimizers, $y^{\varepsilon}\left(  x,t\right)  $ and $y\left(  x,t\right)  $
then they also separated by at most $\tilde{C}\varepsilon.$

Second, we analyze $\left\vert w^{\varepsilon}\left(  x,t\right)  -w\left(
x,t\right)  \right\vert $ for a single minimizer of $tL\left(  \frac{x-y}%
{t}\right)  +g\left(  y\right)  $ and demonstrate that the result can be
extended for a minimum that is attained by multiple, even uncountably many,
minimizers $\left\{  y_{\alpha}\right\}  $.

\subsection{\textbf{Construction of smoothed and uniformly convex }$H$}

Given a function that is locally integrable, one can mollify it using a
standard convolution (\cite{EV}, p. 741). Alternatively, we will use a
mollification as in \cite{GT} in which the difference between a piecewise
linear function and its mollification vanishes outside a small neighborhood of
each vertex.

\begin{lemma}
[Smoothing]\label{Lem Smoothing}Suppose that $G\left(  y\right)  $ is
piecewise linear and satisfies%
\begin{equation}
G^{\prime}\left(  y\right)  =\left\{
\begin{array}
[c]{ccc}%
\alpha<0 & if & y<y_{m}\\
\gamma>0 & if & y>y_{m},
\end{array}
\right.
\end{equation}
$g\left(  y\right)  =\beta y$ for some $\beta \in \left(\alpha, \gamma\right)$ and
$G_{\varepsilon}\left(  y\right)  $ is any function that satisfies%
\begin{equation}
\sup_{y\in A}\left\vert G_{\varepsilon}\left(  y\right)  -G\left(  y\right)
\right\vert \leq C_{1}\varepsilon
\end{equation}
for any given compact set $A.$ Let $y_{m}^{\varepsilon}:=\arg\min\left\{
G_{\varepsilon}\left(  y\right)  -g\left(  y\right)  \right\}  .$ Then one
has$\ y_{m}:=\arg\min\left\{  G\left(  y\right)  -g\left(  y\right)  \right\}
$ and
\begin{equation}
\left\vert y_{m}^{\varepsilon}-y_{m}\right\vert \leq C_{2}\varepsilon
\end{equation}
where $C_{2}$ depends on $A,\alpha,\beta,\gamma$.
\end{lemma}

\begin{proof}
[Proof of Lemma \ref{Lem Smoothing}]The first assertion that $y_{m}$ is the
argmin follows from immediately from the properties assumed for $g$ and
$G^{\prime}.$ To prove the second assertion, i.e., the bound $\left\vert
y_{m}^{\varepsilon}-y_{m}\right\vert \leq C\varepsilon,$ one defines
$\Phi\left(  y\right)  :=G\left(  y\right)  -g\left(  y\right)  $ and
$\Phi_{\varepsilon}\left(  y\right)  :=G_{\varepsilon}\left(  y\right)
-g\left(  y\right)  .$ The graph of $\Phi:=G-g$ is a v-shape with $\Phi\left(
0\right)  =0$. We can draw parallel lines $C\varepsilon$ above and below
$\Phi$ and observe that $\varepsilon\Phi_{\varepsilon}$ lies within these
lines, as illustrated in Figure \ref{LemSFig}.
\begin{figure}
[ptb]
\begin{center}
\includegraphics[width=\linewidth]{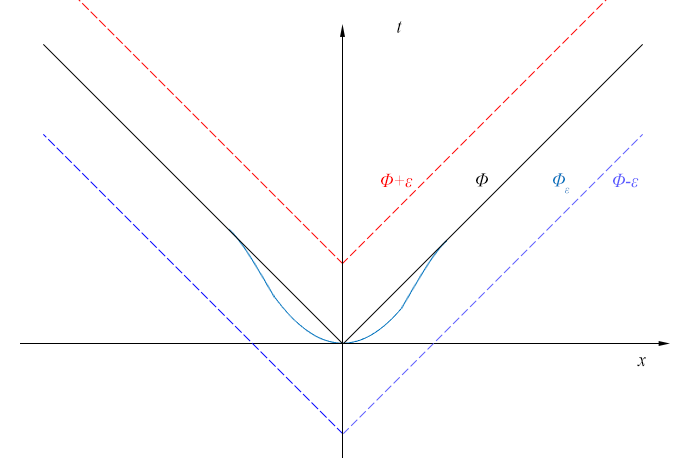}
\caption{Standard mollification of the flux function $\Phi.$ The mollifiation
$\Phi_{\varepsilon}$ is shown in solid light blue, and remains within the
bounds $\Phi\pm\varepsilon$, shown in dashed red and dashed blue. Other
mollifications are also possible, such as in \cite{GT} where the values of
$\Phi_{\varepsilon}$ and $\Phi$ coincide exactly outside an interval of
measure $\varepsilon$.}
\label{LemSFig}
\end{center}
\end{figure}
%EndExpansion

Note that a necessary condition for $y_{m}^{\varepsilon}$ to be the argmin of
$\Phi_{\varepsilon}$ is that
\begin{equation}
\Phi_{\varepsilon}\left(  y_{m}^{\varepsilon}\right)  \leq\Phi_{\varepsilon
}\left(  y_{m}\right)
\end{equation}
since $\Phi_{\varepsilon}\left(  y_{m}^{\varepsilon}\right)  $ must be below
all $\Phi_{\varepsilon}\left(  y\right)  $. Both of the quantities
$\Phi_{\varepsilon}\left(  y_{m}^{\varepsilon}\right)  $ and $\Phi
_{\varepsilon}\left(  y_{m}\right)  $ are within the bounds in the graph above
so that we can write this inequality in the form%
\begin{equation}
\Phi\left(  y_{m}^{\varepsilon}\right)  -C\varepsilon\leq\Phi_{\varepsilon
}\left(  y_{m}^{\varepsilon}\right)  \leq\Phi_{\varepsilon}\left(
y_{m}\right)  \leq\Phi\left(  y_{m}\right)  +C\varepsilon
\end{equation}

Thus $\Phi\left(  y_{m}^{\varepsilon}\right)  -\Phi\left(  y_{m}\right)
\leq2C\varepsilon,$ i.e., by definition of $\Phi$ one has%
\begin{equation}
\left(  \alpha-\beta\right)  \left(  y_{m}^{\varepsilon}-y_{m}\right)
\leq2C\varepsilon.
\end{equation}

So for $y_{m}^{\varepsilon}>y_{m}$ we have the restriction%
\begin{equation}
y_{m}^{\varepsilon}-y_{m}\leq\frac{2C}{\alpha-\beta}\varepsilon.
\end{equation}
In a similar way we obtain a restriction in the other direction and prove the lemma.
\end{proof}

As discussed above, given any $H\left(  q\right)  $ that is piecewise linear
with finitely many break points, one can construct an approximation
$\mathcal{H}_{\varepsilon}\left(  q\right)  $ that has the following properties:

(a) $\left\vert \mathcal{H}_{\varepsilon}\left(  q\right)  -H\left(  q\right)
\right\vert \leq C_{1}\varepsilon$ for all $q\in\mathbb{R}$, and (b)
$\mathcal{H}_{\varepsilon}\left(  q\right)  -H\left(  q\right)  =0$ if
$\left\vert q-c_{i}\right\vert >C_{2}\varepsilon$ \ where $c_{i}$ is any break
point of $H.$

Subsequently, all references to smoothing will mean that conditions (a) and
(b) are satisfied. Note that $H_{\varepsilon}\left(  q\right)  $ is convex
(but not necessarily uniformly convex). Since it is smooth, we have
$\mathcal{H}_{\varepsilon}^{\prime\prime}\left(  q\right)  \geq0.$

In order to have uniform convexity we can add to $\mathcal{H}_{\varepsilon
}\left(  \cdot\right)  $ a term $\delta q^{2}$ where $0<\delta\leq\varepsilon$
and define%
\begin{equation}
\mathcal{H}_{\varepsilon,\delta}\left(  q\right)  :=H_{\varepsilon}\left(
q\right)  +\delta q^{2}.
\end{equation}
Then $\mathcal{H}_{\varepsilon,\delta}^{\prime\prime}\left(  q\right)
>\delta>0$ so that $\mathcal{H}_{\varepsilon,\delta}$ is uniformly convex and
has two continuous derivatives.

\begin{lemma}
[\textbf{Approximate Legendre Transform}]\label{Lem Legendre}Let
\begin{equation}
L\left(  p\right)  :=\sup_{q\in\mathbb{R}}\left\{  pq-H\left(
q\right)  \right\}  ,\ \ \mathcal{L}_{\varepsilon,\delta}\left(  p\right)
:=\sup_{q\in\mathbb{R}}\left\{  pq-\mathcal{H}_{\varepsilon,\delta}\left(
q\right)  \right\}  .
\end{equation}
and let $A$ be a compact set in $\mathbb{R}$. For any $p\in A$ one has
\begin{equation}
\left\vert \mathcal{L}_{\varepsilon,\delta}\left(  p\right)  -L%
\left(  p\right)  \right\vert \leq C\varepsilon.
\end{equation}

\end{lemma}

\begin{remark}
Note that since $H$ has finite max and min slopes, outside this range we have
$\mathcal{L}\left(  p\right)  :=\sup_{q\in\mathbb{R}}\left\{  pq-H\left(
q\right)  \right\}  =\infty$ . For $\mathcal{L}_{\varepsilon,\delta}\left(
p\right)  $ we will have a very large, though not infinite value when $p$
exceeds this range, so $p$ outside this range is also irrelevant in terms of
minimizers. Thus, without loss of generality we can restrict our attention to
$p$ in a compact set.
\end{remark}

\begin{proof}
[Proof of Lemma \ref{Lem Legendre}]Note that $\arg\max_{q}\left\{  pq-H\left(
q\right)  \right\}  =\arg\min_{q}\left\{  H\left(  q\right)  -pq\right\}  $
and similarly for $H_{\varepsilon,\delta}\left(  q\right)  $. We then use the
Lemma above by defining
\[
G_{\varepsilon}\left(  q\right)  :=\mathcal{H}_{\varepsilon,\delta}\left(
q\right)  -pq\ \ with\ \delta:=\varepsilon^{2}<1.
\]
We can then apply the previous Lemma, noting that $\left\vert \mathcal{H}%
_{\varepsilon,\varepsilon^{2}}\left(  q\right)  -H\left(  q\right)
\right\vert \leq C\varepsilon$ implies
\[
\left\vert G_{\varepsilon}\left(  q\right)  -G\left(  q\right)  \right\vert
\leq C\varepsilon\ ,
\]
to conclude that with $q_{m}:=\arg\min G\left(  q\right)  $ and $q_{m}%
^{\varepsilon}:=\arg\min G_{\varepsilon}\left(  q\right)  $%
\[
\left\vert q_{m}-q_{m}^{\varepsilon}\right\vert \leq C\varepsilon.
\]
Since $G$ and $G_{\varepsilon}$ are Lipschitz, one has%
\begin{align*}
\left\vert \mathcal{L}_{\varepsilon,\varepsilon^{2}}\left(  p\right)
-\mathcal{L}\left(  p\right)  \right\vert  &  =\left\vert \sup_{q\in
\mathbb{R}}\left\{  pq-\mathcal{H}_{\varepsilon,\varepsilon^{2}}\left(
q\right)  \right\}  -\sup_{q\in\mathbb{R}}\left\{  pq-H\left(  q\right)
\right\}  \right\vert \\
&  =\left\vert \left\{  pq_{m}^{\varepsilon}-\mathcal{H}_{\varepsilon
,\varepsilon^{2}}\left(  q_{m}^{\varepsilon}\right)  \right\}  -\left\{
pq_{m}-H\left(  q_{m}\right)  \right\}  \right\vert \\
&  \leq\left\vert p\right\vert \ \left\vert q_{m}^{\varepsilon}-q_{m}%
\right\vert +\left\vert \mathcal{H}_{\varepsilon,\varepsilon^{2}}\left(
q_{m}^{\varepsilon}\right)  -H\left(  q_{m}\right)  \right\vert \\
&  \leq\left\vert p\right\vert \ \left\vert q_{m}^{\varepsilon}-q_{m}%
\right\vert +\left\vert \mathcal{H}_{\varepsilon,\varepsilon^{2}}\left(
q_{m}^{\varepsilon}\right)  -H\left(  q_{m}^{\varepsilon}\right)  \right\vert
\\
&  +\left\vert H\left(  q_{m}^{\varepsilon}\right)  -H\left(  q_{m}\right)
\right\vert \\
&  \leq\tilde{C}\varepsilon
\end{align*}
\newline since we noted above that $\left\vert \mathcal{H}_{\varepsilon
,\varepsilon^{2}}\left(  q\right)  -H\left(  q\right)  \right\vert \leq
C\varepsilon$ and $H$ is Lipschitz. This proves the Lemma.
\end{proof}

\subsection{Proving Convergence of Solutions $w^{\varepsilon}\rightarrow w$}

We define%
\[
f\left(  y;x,t\right)  =tL\left(  \frac{x-y}{t}\right)  ,\ \ f_{\varepsilon
,\delta}\left(  y;x,t\right)  =t\mathcal{L}_{\varepsilon,\delta}\left(
\frac{x-y}{t}\right)
\]
and consider any fixed $t\in\left[  0,T\right]  $ for some $T\in\mathbb{R}%
^{+}$ and $x,y$ on bounded intervals. The bounds on $\mathcal{L}$ and
$\mathcal{L}_{\varepsilon,\delta}$ then imply
\[
\left\vert f\left(  y;x,t\right)  -f_{\varepsilon,\delta}\left(  y;x,t\right)
\right\vert \leq C\varepsilon.
\]
Next, we claim that if there is a single minimizer for $f\left(  y;x,t\right)
+g\left(  y\right)  $, denoted by $y^{\ast}\left(  x,t\right)  $ and
$y_{\varepsilon}^{\ast}\left(  x,t\right)  $ for $f_{\varepsilon}\left(
y;x,t\right)  +g\left(  y\right)  $, then $y_{\varepsilon}^{\ast}\left(
x,t\right)  \rightarrow y^{\ast}\left(  x,t\right)  $ a.e. in $x$. This is
proven in the same way as Lemma \ref{Lem Legendre}, and analyzed in Section 4.
Note that if the minimum is not within $C\varepsilon$ of the vertex (i.e. not
near a minimum of $g$), then the conclusion will be immediate since the
mollification does not extend more than a distance $C\varepsilon$ from the vertex.

Now, we want to compare the solution $w\left(  x,t\right)  $ with
$w^{\varepsilon}\left(  x,t\right)  $ and assert that for any $t>0$ and a.e.
$x$ one has
\[
\lim_{\varepsilon\rightarrow0}w^{\varepsilon}\left(  x,t\right)  =w\left(
x,t\right)  .
\]
If we assumed that there is a single minimizer, $y^{\ast}\left(  x,t\right)  $
then the result would be clear from the relations%
\[
w\left(  x,t\right)  =g^{\prime}\left(  y^{\ast}\left(  x,t\right)  \right)
,\ w^{\varepsilon}\left(  x,t\right)  =g^{\prime}\left(  y_{\varepsilon}%
^{\ast}\left(  x,t\right)  \right)  ,\
\]
and the fact that $g^{\prime}$ is continuous in $x$ and $y_{\varepsilon}%
^{\ast}\left(  x,t\right)  \rightarrow y^{\ast}\left(  x,t\right)  $.

The subtlety is when we have more than one minimizer. Note from the earlier
material on the sharp problem, we only need to consider finitely many
minimizers, since there are only finitely many vertices, and only finitely
many segments of $L.$ The minimizers can be at the vertices, or they may be on
the segments. But if they are on the segments, the minimizers are in finitely
many classes that correspond to the same value of $g^{\prime}\left(  \hat
{y}\left(  x,t\right)  \right)  =L^{\prime}\left(  \frac{x-\hat{y}\left(
x,t\right)  }{t}\right)  .$ Thus we can choose the larger of these two
minimizers, for example.

We consider the two Types $\left(  A\right)  $ and $\left(  B\right)  $ and
suppose first that there are two minimizers of the same type.

\textbf{Type (A)}. Suppose that for some $\left(  x_{0},t\right)  $ we have
$\hat{y}_{1}$ and $\hat{y}_{2}$ that are both Type $\left(  A\right)  $
minimizers, i.e., at different vertices of $L$. If $g^{\prime}\left(  \hat
{y}_{1}\right)  <$ $g^{\prime}\left(  \hat{y}_{2}\right)  $ then at some
$\varepsilon,$ the mollified versions will also satisfy $g^{\prime}\left(
\hat{y}_{1}^{\varepsilon}\right)  <$ $g^{\prime}\left(  \hat{y}_{2}%
^{\varepsilon}\right)  .$ This means that for $x\not =x_{0}$ only $\hat{y}%
_{1}$ and $\hat{y}_{1}^{\varepsilon}$ will be relevant. Analogously, we have
the opposite inequality. If we have $g^{\prime}\left(  \hat{y}_{1}\right)  =$
$g^{\prime}\left(  \hat{y}_{2}\right)  $, then we may not have $g^{\prime
}\left(  \hat{y}_{1}^{\varepsilon}\right)  =$ $g^{\prime}\left(  \hat{y}%
_{2}^{\varepsilon}\right)  .$ However, since $g^{\prime}$ is continuous, and
we know that
\[
\hat{y}_{1}^{\varepsilon}\rightarrow\hat{y}_{1}\ \ and\ \ \hat{y}%
_{2}^{\varepsilon}\rightarrow\hat{y}_{2}%
\]
we have then%
\[
g^{\prime}\left(  \hat{y}_{1}^{\varepsilon}\right)  \rightarrow g^{\prime
}\left(  \hat{y}_{1}\right)  \ \ and\ \ g^{\prime}\left(  \hat{y}%
_{2}^{\varepsilon}\right)  \rightarrow g^{\prime}\left(  \hat{y}_{2}\right)
\ .
\]
Thus we have from our basic results $w\left(  x,t\right)  =g^{\prime}\left(
y^{\ast}\left(  x,t\right)  \right)  $ and $w^{\varepsilon}\left(  x,t\right)
=g^{\prime}\left(  y_{\varepsilon}^{\ast}\left(  x,t\right)  \right)  $, that%
\[
w^{\varepsilon}\left(  x,t\right)  \rightarrow w\left(  x,t\right)  .
\]
In Figure \ref{MinimizersFig}(a), we illustrate the case where $g^{\prime
}\left(  y_{1}\right)  =g^{\prime}\left(  y_{2}\right)  .$

\textbf{Type (B)}$.$ In this case we are on the straight portion of $L,$ and
the minimum must occur (as discussed in the earlier section) when
$\partial_{y}\left\{  L\left(  \frac{x_{0}-y}{t}\right)  +g\left(  y\right)
\right\}  =0.$ This means that any minimum $\hat{y}$ must satisfy $L^{\prime
}\left(  \frac{x_{0}-\hat{y}}{t}\right)  =g^{\prime}\left(  \hat{y}\right)  $.
If there are two minima at different segments, the one corresponding to the
lowest value of $g^{\prime}$ will be relevant. To see this, recall from Figure
\ref{PWF} that we start with break points \ $c_{1}<...<c_{N+1}$ and
$c_{1}<0,\ c_{N+1}>0.$ These become the slopes for $L$ and, when we define
$f\left(  y;x,t\right)  :=tL\left(  \frac{x-y}{t}\right)  $ the slopes are
$-c_{N+1}<0$ on the left up to the last one, $-c_{1}>0$ on the right. Any
minimizer, $\hat{y},$ of
\[
f\left(  y;x,t\right)  +g\left(  y\right)  =tL\left(  \frac{x-y}{t}\right)
+g\left(  y\right)
\]
on the flat part of $f$ must satisfy $f^{\prime}\left(  \hat{y}\right)
+g^{\prime}\left(  \hat{y}\right)  =0.$ On the last segment, for example we
have the requirement%
\[
g^{\prime}\left(  \hat{y}\right)  =-\left(  -c_{1}\right)  =c_{1}<0.
\]
For any of the previous segments we obtain $g^{\prime}\left(  y\right)
=c_{j}>c_{1}$ . Hence if there are minimizers on previous segments, they
become irrelevant as soon as we increase $x$.
\begin{figure}
[ptb]
\begin{center}
\includegraphics[width=\linewidth]{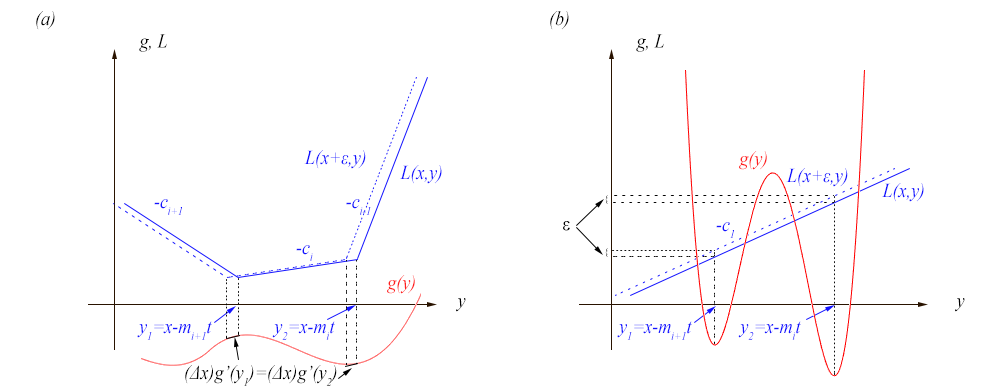}
\caption{(a) For multiple minimizers of\ Type (A), if $g^{\prime}$ evaluated
at different minimzers has a different value, as soon as $x$ is varied
slightly the value of $L\left(  \frac{x-y}{t}\right)  +g\left(  y\right)  $
changes proportionally to the slope $c_{i}$, so all but one of the minimizers
cease to be a minima as they change by different amounts. Therefore this case
occurs for a set of $x$ having measure zero. If $g^{\prime}$ is equal at two
or more minimizers (pictured), then the change of the value is uniform to
first order and we may consider the minimizer with largest argument without
loss of generality; (b) By a similar argument for minimizers of Type (B) we
need only consider those on the last segment, with slope $-c_{1}$. For all
minimizers $y^{\ast}$ of this form, $g^{\prime}\left(  y^{\ast}\left(
x,t\right)  \right)  $ will be identical, so choose the greatest without loss
of generality.}
\label{MinimizersFig}
\end{center}
\end{figure}
%EndExpansion

The slope of the straight line, $f,$ will be positive, i.e., $-c_{1}$ and the
minimum in the illustration above will be just to the left of the minimum of
$g$.

Hence, in the case of a minimum of Type $\left(  B\right)  $ we see that it is
only the minimizers on this rightmost segment that are relevant (except on a
set of measure zero). Although there may be infinitely many minimizers on this
segment, they all yield the same $g^{\prime}\left(  \hat{y}_{j}\right)  $
value of $c_{k}$ (where $k$ is the minimum index for which the slope $c_{k}$
corresponds to a minimizer), so we can take the largest of them, $y^{\ast}$
and write $w\left(  x,t\right)  =g^{\prime}\left(  y^{\ast}\left(  x,t\right)
\right)  =c_{k}$. Thus one has $w\left(  x,t\right)  =c_{k}$. This is
illustrated in Figure \ref{MinimizersFig}(b).

If we have a combination of minimizers of the two types, then the situation is
similar. Except on a set of measure zero in $x$, we need only consider those
values for which $g^{\prime}\left(  \cdot\right)  $ is minimum, the other
values cease to become a minimum as $x$ is varied. As discussed above, we only
need to consider finitely many of these minimizers.

Estimating $\left\vert w^{e}-w\right\vert $ is simpler in the Type $\left(
B\right)  $ case since we only have finitely many values for $g^{\prime}.$
When we take the smoothed version, $\mathcal{H}^{\varepsilon}$ yielding the
smoothed $\mathcal{L}^{\varepsilon}$ each of these points $\hat{y}_{j}\left(
x,t\right)  $ is approximated by $\hat{y}_{j}^{\varepsilon}\left(  x,t\right)
$ that will correspond to the same $g^{\prime}.$ Note that on the flat part
(non-vertex) of $L,$ the smoothing in the way that we are doing it does not
change the slope. In fact, $L$ and $L^{\varepsilon}$ will be identical except
on an interval of order $\varepsilon$ about the vertices.

Once have isolated the minimizer, $y^{\ast},$ we have that the $y_{\varepsilon
}^{\ast}$ is within $\varepsilon$ of $y^{\ast}$. Previous results then yield
the convergence of $w^{\varepsilon}\left(  x,t\right)  $ to $w\left(
x,t\right)  .$We will also need the following technical lemma.

\begin{lemma}
\label{LemMinConv}For the conservation law with the smoothed flux function
$\mathcal{H}_{\varepsilon,\delta},$ defined $A_{\delta,\varepsilon}$ as the
set of minimizers $y\left(  x,t\right)  $ in the Hopf-Lax formula. Similarly,
define $A$ as the set of minimizers for the sharp problem with the piecewise
linear flux function $H$. Then we have%
\begin{equation}
\lim_{\varepsilon\downarrow0}\sup A_{\varepsilon,\delta}=\sup A\text{ a.e.}%
\end{equation}

\end{lemma}

\begin{theorem}
\label{ThmMinConv}For given $t>0$ and for a.e. $x$, the largest minimizer
$y^{\ast}\left(  x,t\right)  $ of the sharp problem satisfies the identity%
\begin{equation}
w\left(  x,t\right)  =g^{\prime}\left(  y^{\ast}\left(  x,t\right)  \right)
\end{equation}
by Theorem \ref{Thm 4.3}. In particular, this implies that%
\begin{equation}
\lim_{\varepsilon\rightarrow0}w^{\varepsilon}\left(  x,t\right)  =w\left(
x,t\right)  \label{ThmMinConvEq}%
\end{equation}
pointwise a.e.

\begin{proof}
[Proof of Theorem \ref{ThmMinConv}]Should the set $A_{\varepsilon,\delta}$
consist of a single element, the proof is trivial. Therefore, assume that the
minimizer in $A_{\varepsilon,\delta}$ is not unique. We may consider without
loss of generality the case of two such minimizers, as the arguments presented
here are easily generalized to $n$ such minimizers.

We consider two such minimizers at a point $x_{0}$ for the sharp problem and
denote them by $y_{i}\left(  x_{0}\right)  $ ($i=1,2$). First take the case
where they are both Type (A) minimizers. We take the partial derivative with
respect to $x$ of the Legendre transform at the minimum point, which is
well-defined as $y_{1}$ also depends on $x$ and the minimizer moves as one
shifts $x$. Therefore, one has%
\begin{equation}
g^{\prime}\left(  y_{1}\left(  x_{0}\right)  \right)  =\partial_{x}\left\{
tL\left(  \frac{x-y_{1}\left(  x\right)  }{t}\right)  +g\left(  y_{1}\left(
x\right)  \right)  \right\}  _{x=x_{0}}%
\end{equation}
and similarly for $g^{\prime}\left(  y_{2}\left(  x_{0}\right)  \right)  .$ If
$g^{\prime}\left(  y_{1}\left(  x_{0}\right)  \right)  \not =g^{\prime}\left(
y_{2}\left(  x_{0}\right)  \right)  $, the minimizer with the smaller value
when evaluated under $g^{\prime}$ will become irrelevant as $x_{0}$ changes.
Therefore, one sees that this case is confined to sets of measure zero and can
be ignored.

Consequently, assume
\begin{equation}
g^{\prime}\left(  y_{1}\left(  x_{0}\right)  \right)  =g^{\prime}\left(
y_{2}\left(  x_{0}\right)  \right)  . \label{SmoothMin}%
\end{equation}
We now want to examine $y_{1}^{\varepsilon}$ and $y_{2}^{\varepsilon}$, the
minimizers for the smoothed out version. We have suppressed the parameter
$\delta$ (by setting $\delta=\varepsilon^{2}$) and the time $t$ for notational
convenience. We then compute%
\begin{equation}
\partial_{x}\left\{  t\mathcal{L}\left(  \frac{x-y_{i}^{\varepsilon}\left(
x_{0}\right)  }{t}+g\left(  y_{i}^{\varepsilon}\left(  x_{0}\right)  \right)
\right)  \right\}  =g^{\prime}\left(  y_{i}^{\varepsilon}\left(  x_{0}\right)
\right)  .
\end{equation}
It is possible that one might have $g^{\prime}\left(  y_{1}^{\varepsilon
}\left(  x_{0}\right)  \right)  <g^{\prime}\left(  y_{2}^{\varepsilon}\left(
x_{0}\right)  \right)  $ (or the reverse inequality) despite having
(\ref{SmoothMin}). However, this would imply that for the $\varepsilon$ case,
$y_{2}^{\varepsilon}$becomes irrelevant due to $g\left(  y_{2}^{\varepsilon
}\right)  $ increasing faster as $x$ is changed, and in this case
$y_{1}^{\varepsilon}$ would be left as the largest minimizer. However,
$g^{\prime}$ is continuous, so that%
\begin{align}
g^{\prime}\left(  y_{1}^{\varepsilon}\left(  x_{0}\right)  \right)   &
\rightarrow g^{\prime}\left(  y_{1}\left(  x_{0}\right)  \right) \nonumber\\
g^{\prime}\left(  y_{2}^{\varepsilon}\left(  x_{0}\right)  \right)   &
\rightarrow g^{\prime}\left(  y_{2}\left(  x_{0}\right)  \right)  .
\label{Sec5MinConv}%
\end{align}
Hence, (\ref{Sec5MinConv}) implies $g^{\prime}\left(  y_{1}^{\varepsilon
}\left(  x_{0}\right)  \right)  \rightarrow g^{\prime}\left(  y_{2}\left(
x\right)  \right)  $. Note that this result holds even though one may have
$y_{1}^{\varepsilon}\not \rightarrow y_{2}$.

Thus, if there were more than two Type (A) minimizers, then we would have%
\begin{equation}
w\left(  x,t\right)  =g^{\prime}\left(  y^{\ast}\left(  x,t\right)  \right)
=\lim_{\varepsilon\rightarrow0}w^{\varepsilon}\left(  x,t\right)
\end{equation}

Next, consider the situation with more than one Type (B) minimizer. Should
these minimizers occur among points where $L^{\prime}\left(  \cdot\right)  $
takes different values, the one with the smaller value evaluated at
$L^{\prime}$ becomes irrelevant by the same token as for multiple Type
(A)\ minimizers. Therefore, assume without loss of generality that these
minimizers both occur along the last segment, i.e., where $f$ has slope
$-c_{1}$%
\begin{equation}
g^{\prime}\left(  \hat{y}_{i}\left(  x_{0}\right)  \right)  =c_{1}.
\end{equation}

Then we have%
\begin{equation}
\partial_{x}\min\left\{  {}\right\}  =\partial_{x}\left\{  tL\left(
\frac{x-y_{i}\left(  x\right)  }{t}\right)  +g\left(  y_{i}\left(  x\right)
\right)  \right\}  =-c_{1}%
\end{equation}
so that all derivatives will be identical.

Next, for each $y_{i}$we have $y_{i}^{\varepsilon}$ such that $\left\vert
y_{i}^{\varepsilon}-y_{i}\right\vert <C\varepsilon$, so one may write%
\begin{equation}
\partial_{x}\min_{y\in\mathbb{R}}\left\{  tL\left(  \frac{x-y}{t}+g\left(
y\right)  \right)  \right\}  =-c_{N}=\lim_{\varepsilon\rightarrow0}g^{\prime
}\left(  y_{\varepsilon}^{\ast}\left(  x,t\right)  \right)
\end{equation}
where%
\begin{equation}
y_{\varepsilon}^{\ast}:=\arg^{+}\min_{y\in\mathbb{R}}\left\{  t\mathcal{L}%
\left(  \frac{x-y}{t}\right)  +g\left(  y\right)  \right\}  .
\end{equation}
This leaves the one remaining case where there is one minimizer of each type,
i.e. one Type (A) minimizer and one Type (B). As in the above cases, one can
then assume that $g^{\prime}$ has the same value at both of these minimizers,
for if not, the minimizer with the greater value of $g^{\prime}$ would cease
to be relevant, so that the set of such $x$ for fixed $t$ where this occurs is
of measure zero. In a similar fashion to the above cases, one has
(\ref{ThmMinConvEq}). Together with Lemma \ref{LemMinConv}, this completes the proof.
\end{proof}
\end{theorem}

\section{Uniqueness For Polygonal Flux}

In the preceding sections, we have shown that $w\left(  x,t\right)  =g'\left(
y^{\ast}\left(  x,t\right)  \right)  $ is a solution to the conservation law
(\ref{cl}). In this section we establish a criterion under which it
is the only solution by characterizing $w$ as the unique solution constructed from the limit of the
functions $w^{\varepsilon}$as $\varepsilon\downarrow0$, which are unique
provided $g^{\prime}$ is continuous. This approach is reminiscent of
the well-known vanishing viscosity limit for Burgers' equation.

\begin{definition}
\label{Def 6.1}Let $H\in C^{0}\left(  \mathbb{R}\right)  $ and suppose further
that $H$ is differentiable a.e. We say $w\left(  x,t\right)  $ is a
\textit{limiting mollified solution} to the initial value problem for the
conservation law for the flux function $H$ if

(i) There exist smooth $\mathcal{H}_{j}$ that converge uniformly on compact
sets to $H$.

(ii) The solutions $w_{j}$ for the conservation law with flux $\mathcal{H}%
_{j}$ converge, for each $t>0$, to $w$ a.e. in $x$.

(iii) For any other sequence $\mathcal{\tilde{H}}_{j}$ and solutions
$\tilde{w}_{j}$ satisfying (i) and (ii), we have%
\begin{equation}
\lim_{j\rightarrow\infty}\tilde{w}_{j}=\lim_{j\rightarrow\infty}w_{j}=w.
\end{equation}

\end{definition}

\begin{remark}
For the case of a polygonal flux $H$ with break points $\left\{
c_{i}\right\}  _{i=1}^{n}$, clearly $H\in C^{\infty}\left(  \mathbb{R}%
\backslash\left\{  c_{i}\right\}  _{i=1}^{n}\right)  $ and is continuous on
all of $\mathbb{R}$. Indeed, we can show rigorously that this case satisfies
Definition \ref{Def 6.1}.
\end{remark}

\begin{theorem}
\label{Thm 6.3}Let $g^{\prime}$ be continuous, $H$ a polygonal flux
function, and $w$ be the solution of the corresponding conservation law. Then%
\begin{equation}
w\left(  x,t\right)  =g^{\prime}\left(  y^{\ast}\left(  x,t\right)  \right)
\end{equation}
\bigskip is the unique limiting mollified solution satisfying $w\left(  x,0\right)
=g^{\prime}\left(  x\right)  $.
\end{theorem}

\begin{proof}
[Proof of Theorem \ref{Thm 6.3}]Since $w_{j}$ are weak solutions and
$w_{j}\rightarrow w$ for any sequence $w_{j}$ (as shown in Section 5), the
result follows.
\end{proof}

\section{A Discretized Conservation Law: Polygonal Flux with Matching
Piecewise Constant Initial Conditions}

In earlier sections, when considering the piecewise linear flux function $H$,
we chose initial conditions that were smooth. An important version of this
problem deals with initial conditions $g^{\prime}$ that are not smooth, but
instead piecewise constant, as is treated in \cite{HO}. The values of these
constants are taken as a subset of the break points $\left\{  c_{i}\right\}
_{i=1}^{N}$ of $H$. Consequently, as one can see from Figure \ref{PWF}, the
values of the initial condition match the slopes of the Legendre transform $L$
of the flux function. Furthermore, direct computation verifies that the range
of the solution $w\left(  x,t\right)  $ will also have range $\left\{
c_{i}\right\}  _{i=1}^{N}$.

The analysis of the minimizers is similar to those of the previous sections,
except from the fact that we have an additional type of minimizer, Type (C) in
which the vertices of $f\left(  y;x,t\right)  :=L\left(  \frac{x-y}{t}\right)
$ and $g^{\prime}$ coincide For such a minimum, the derivative $g^{\prime
}\left(  \hat{y}\left(  x,t\right)  \right)  $ does not exist as in general
the limits from the left and right do not agree. However, there are only
finitely many vertices of $L$, and hence there are at most a finite number of
Type (C) vertices for a fixed $t$.

For Type (A)\ and (B) minimizers, we proceed in the same way, including the
smoothing. Although the Type (B) minimum now occurs at a vertex of $g$, the
analysis of the $x$-derivative yields the same result. Note that in this case
one also needs to mollify $g^{\prime}$, yielding the following results.

\begin{theorem}
\label{Thm 7.1}Let $L$ be polygonal convex and $g^{\prime}$ be piecewise
constant. Then
\begin{align}
w\left(  x,t\right)    & =g^{\prime}\left(  y^{\ast}\left(  x,t\right)
\right)  \text{ when the minimizer }y^{\ast}\text{ is at a vertex of
}L\nonumber\\
& =L^{\prime}\left(  y^{\ast})x,t\right)  \text{ when the minimizer }y^{\ast
}\text{ is at a vertex of }g \label{Thm7.1cases}
\end{align}
a.e. in $x$ (for fixed $t>0$) is a solution to (\ref{cl}). Note that for a
fixed $t>0$ and $x$ a.e., one of the cases in (\ref{Thm7.1cases}) occurs.
Furthermore%
\begin{equation}
w^{\varepsilon}\left(  x,t\right)  =g_{\varepsilon}^{\prime}\left(
y_{\varepsilon}^{\ast}\left(  x,t\right)  \right)  \rightarrow w\left(
x,t\right)  \text{ a.e. in }x
\end{equation}

\end{theorem}

Note that one can apply the limiting mollified uniqueness concept in the same
manner as earlier.

\begin{proof}
[Proof of Theorem \ref{Thm 7.1}]To obtain the result $w^{\varepsilon
}\rightarrow w$, we observe that if $h:=g^{\prime}$ is piecewise constant,
then $g\left(  y\right)  :=\int_{0}^{y}h\left(  s\right)  ds$ is Lipschitz
with $Lip\left(  g\right)  =\max_{i}\left\{  \left\vert c_{i}\right\vert
\right\}  $, and differentiable a.e by Rademacher's Theorem.

Note that the only subtlety is for Type (C) in which the minimizer of
$tL^{\varepsilon}\left(  \frac{x-y}{t}+g_{\varepsilon}\left(  y\right)
\right)  $ may be on one segment of $g$ for which the minimizer $tL\left(
\frac{x-y}{t}\right)  +g\left(  y\right)  $ is on the adjacent one. But this
is an issue that is of measure $0$ in $x$ for a given $t$.
\end{proof}

When restricting the values of the initial conditions to the break points of
$H$, we obtain the following more specific result.

\begin{corollary}
\label{Cor 7.2}If $L$is polygonal convex with break points $\left\{
c_{i}\right\}  _{i=1}^{N}$ and the range of $g$ is contained in $\left\{
c_{i}\right\}  $, then the solution $w\left(  x,t\right)  $ takes on values
only in $\left\{  c_{i}\right\}  .$
\end{corollary}

\begin{proof}
[Proof of Corollary \ref{Cor 7.2}]The minimizers of $tL\left(  \frac{x-y}%
{t}\right)  +g\left(  y\right)  $ will consist of the vertices of $g$ and
$tL\left(  \frac{x-y}{t}\right)  $ exclusively. If $\hat{y}$ is a Type (A)
minimizer (i.e. vertex of $L$ but on the differentiable portion of $g$), then%
\begin{equation}
\partial_{x}\left\{  tL\left(  \frac{x-\hat{y}\left(  x,t\right)  }{t}\right)
+g\left(  \hat{y}\left(  x,t\right)  \right)  \right\}  =g^{\prime}\left(
\hat{y}\left(  x,t\right)  \right)
\end{equation}
as before. If it is of Type (B), i.e. $\hat{y}$ is at a vertex of $tL\left(
\frac{x-y}{t}\right)  $, then $\partial_{x}\left\{  ...\right\}  =L^{\prime
}\left(  \frac{x-y}{t}\right)  $. In both cases, $\partial_{x}\left\{
...\right\}  \in\left\{  c_{i}\right\}  $. Hence, this is an alternative proof
of \cite{HO}, p. 74.
\end{proof}

\section{Conclusions And Applications}

In this paper, we have shown a number of important extensions to classical
results. In the classical Lax-Oleinik theory, more restrictive assumptions
such as $C^{2}$ smoothness and \textit{uniform} convexity of the flux function
are required. In many of our results, we have proven rigorous theorems with
only a $C^{0}$, (non-strictly) convex flux function $H$. This is particularly
significant as it facilitates an understanding of the behavior introduced by sharp
corners, i.e. at points where the flux function fails to have a derivative in
the classical (non-weak) sense and is nowhere strictly convex.

In fact, when the assumptions mentioned above are relaxed, the uniqueness of
the minimizers does not, in general, persist. Indeed, there is the potential
to have the minimum achieved at an infinite, even uncountable number of
points. However, we have shown that this difficulty can be addressed by
considering the greatest of these minimizers $y^{\ast}\left(  x,t\right)  $,
or supremum in the case of an infinite number. We have shown that the solution
is described by $w\left(  x,t\right)  =g^{\prime}\left(  y^{\ast}\left(
x,t\right)  \right)  $, so we have effectively substituted the requirement for
uniqueness of the minimizer with the behavior of a specific, well-defined
element of the set of minimizers after analyzing the relative change of the
Hopf-Lax functional at each of these points. The results have immediate
application to conservation laws subject to stochastic processes. For example,
if the initial condition $g^{\prime}\left(  x\right)  $, is assumed to be
Brownian motion, then the solution at time $t>0$ is given by $w\left(
x,t\right)  =g^{\prime}\left(  y^{\ast}\left(  x,t\right)  \right)  \,$. In
the case of Brownian motion \cite{A, BE, SC} with fixed value $0$ at $x=0$,
one obtains that the mean and variance at $t$ are $0$ and $y^{\ast}\left(
x,t\right)  $, respectively.

For each $t>0$, we know that $y^{\ast}\left(  x,t\right)  $ is an increasing
function of $x$ from Theorem \ref{Thm 4.2}. Since the variance of Brownian motion also
increases as $\left\vert x\right\vert $ increases, we obtain the result that
the increase in variance persists for all time.

This is an example of the application of these results to random initial
conditions. The methodology can also provide a powerful computational tool.
Computing solutions of shocks from conservation laws is a complicated task
even when the initial data are regular. When one has random initial data, e.g.
Brownian motion (or even less regular randomness), the difficulties are compounded.

The results we have obtained suggest a computational method that amounts to
determining the minimum for the function $tL\left(  \frac{x-y}{t}\right)
+g\left(  y\right)  $. In this expression, the first term can be regarded as a
deterministic slope while the second is an integrated Brownian motion that can
easily be approximated by a discrete stochastic process. In this way one can
obtain the probabilistic features of the solution $w\left(  x,t\right)  $
without tracking and maintaining the shock statistics. In a future paper, we
plan to address in detail the application of these results to an array of
stochastic processes.

\bibliographystyle{plain}
\bibliography{mybib}

\end{document}